\documentclass[preprint,3p]{elsarticle}
\usepackage[latin1]{inputenc}
\usepackage{amsfonts}
\usepackage{enumitem}
\usepackage{color}
\usepackage{bm}
\usepackage[T1]{fontenc}
\usepackage[french,english]{babel}
\usepackage{graphicx,dsfont}
\usepackage{amssymb,amsmath,amsthm,}

\newtheorem{define}{Definition}[section]
\newtheorem{pro}{Proposition}[section]
\newtheorem{Lem}{Lemma}[section]
\newtheorem{cor}{Corollary}[section]

\newtheorem{remark}{Remark}[section]
\theoremstyle{plain} \newtheorem{thm}{Theorem}[section]
\let\div\undefined
\DeclareMathOperator{\div}{div}

\DeclareMathOperator{\ess}{ess}

\catcode`\@=11
\makeatletter

\@addtoreset{equation}{section}

\newenvironment{abstracts}
 {\global\setbox\absbox=\vbox\bgroup
    \hsize=\textwidth
    \linespread{1}\selectfont}
 {\vspace{-\bigskipamount}\egroup}
\renewenvironment{abstract}[1][]
 {\if\relax\detokenize{#1}\relax\else\selectlanguage{#1}\fi
  \noindent\textbf{\abstractname}\par\medskip\noindent\ignorespaces}
 {\par\bigskip}

\begin{document}

\title{Strong solutions of the double phase parabolic equations with variable growth}
\author[1,3]{Rakesh Arora\fnref{fn1}}
 \ead{arora.npde@gmail.com, arora@math.muni.cz}
 \fntext[fn1]{The first author acknowledges the support of the Research Grant from Czech Science Foundation, project GJ19-14413Y for the second part of this work.}

\author[Address02]{Sergey Shmarev\fnref{fn2}}
 \ead{shmarev@uniovi.es}
  \fntext[fn2]{The second author acknowledges the support of the Research Grant MTM2017-87162-P,
    Spain.}
    \fntext[fn2]{Corresponding author}

\address[1]{LMAP, UMR E2S-UPPA CNRS 5142 B\^atiment IPRA, Avenue de l'Universit\'e F-64013 Pau, France
}
\address[3]{Department of Mathematics ans Statistics, Masaryk University, Brno,
Czech Republic.

}
\address[Address02]{Mathematics Department, University of Oviedo, c/Federico Garc\'{i}a Lorca 18, 33007, Oviedo, Spain}

\begin{abstracts}

\begin{abstract}
  This paper addresses the questions of existence and uniqueness of strong solutions to the homogeneous Dirichlet problem for the double phase equation with operators of variable growth:

\[
u_t - \div\left(|\nabla u|^{p(z)-2} \nabla u+ a(z) |\nabla u|^{q(z)-2} \nabla u \right) = F(z,u) \quad \text{in $Q_T=\Omega \times (0,T)$}
\]
where $\Omega \subset \mathbb{R}^N$, $N \geq 2$, is a bounded domain with the smooth boundary $\partial\Omega$, $z=(x,t)\in Q_T$, {$a:\overline Q_T \mapsto \mathbb{R}$} is a given nonnegative coefficient, and the nonlinear source term has the form
\[
F(z,v)=f_0(z)+b(z)|v|^{\sigma(z)-2}v.
\]
The variable exponents {$p$, $q$, $\sigma$} are given functions defined on $\overline{Q}_T$, {$p$, $q$} are Lipschitz-continuous and
\[
\dfrac{2N}{N+2}<p^-\leq p(z) \leq q(z) < p(z) + {\frac{r}{2}} \quad \text{with $0<r<r^\ast=\frac{4p^-}{2N + p^-(N+2)}$,\quad $p^-=\min_{\overline{Q}_T}p(z)$}.
\]
The initial function $u_0$ belongs to a Musielak-Sobolev space associated with the flux. We find conditions on the functions {$f_0$, $a$, $b$, $\sigma$} sufficient for the existence of a unique strong solution with the following global regularity and integrability properties:
\[
\begin{split}
& u_t \in L^{2}(Q_T),\quad |\nabla u|^{p(z)+\delta}\in L^1(Q_T)\quad \text{for every  $0<\delta< r^*$},
\\
& \text{$|\nabla u|^{s(z)},\,a(z)|\nabla u|^{q(z)} \in L^{\infty}(0,T;L^1(\Omega))$ with $ s(z)=\max\{2,p(z)\}$}
.
\end{split}
\]
The same results are established for the equation with the regularized flux

\[
(\epsilon^2+|\nabla u|^2)^{\frac{p(z)-2}{2}}\nabla u + a(z) (\epsilon^2+|\nabla u|^2)^{\frac{q(z)-2}{2}}\nabla u, \qquad  \epsilon>0.
\]
\end{abstract}

\end{abstracts}

%\today

\maketitle

%\tableofcontents

\section{Introduction}
Let $\Omega \subset \mathbb{R}^N$ be a smooth bounded domain, $N \geq 2$ and $0<T< \infty$. We consider the following parabolic problem with the homogeneous Dirichlet boundary conditions:
\begin{equation}\label{eq:main}
         \begin{cases}
             & u_t - \div\left(|\nabla u|^{p(z)-2} \nabla u + a(z) |\nabla u|^{q(z)-2} \nabla u\right)
             = F(z,u)
             \quad\mbox{in $Q_T$},
             \\
             & \mbox{$u=0$ on $\Gamma_T$},
             \\
             & \mbox{$u(x,0)=u_0(x)$ in $\Omega$},
          \end{cases}
\end{equation}
where $z=(x,t)$ denotes the point in the cylinder $Q_T=\Omega \times (0,T]$ and $\Gamma_T= \partial\Omega \times (0,T)$ is the lateral boundary of the cylinder. The nonlinear source has the form
\begin{equation}
\label{eq:source}
F(z,v)=f_0(z)+b(z)|v|^{\sigma(z)-2}v.
\end{equation}
Here {$a \geq 0$, $b$, $p$, $q$, $\sigma$ and $f_0$} are given functions of the variables $z\in Q_T$.

Equations of the type \eqref{eq:main} are often termed ``the double phase equations''. This name, introduced in \cite{colombo2015bounded,colombo2015regularity}, reflects the fact that the flux function $(|\nabla u|^{p(z)-2} + a(z)|\nabla u|^{q(z)-2}) \nabla u$ includes two terms with different properties.  If $p(z)\leq q(z)$ a.e. in the problem domain and $a(z)$ is allowed to vanish  on a set of nonzero measure in $Q_T$, then the growth of the flux is determined by $p(z)$ on the set where $a(z)=0$, and by $q(z)$ wherever $a(z)>0$.

\subsection{Previous work}

The study of the double phase problems started in the late 80th by the works of V.~Zhikov \cite{zhikov-1987,Zhikov-1995} where the models of strongly anisotropic materials were considered  in the context of homogenization. Later on, the double phase functionals
\[
u \to \int_{\Omega} \left( |\nabla u|^p + a(x) |\nabla u|^q\right) ~dx
\]
attracted attention of many researchers. On the one hand, the study of these functionals is a challenging mathematical problem. On the other hand, the double phase functionals appear in a variety of physical models. We refer here to \cite{ball-1976, zhikov-2011} for  applications in the elasticity theory, \cite{bahrouni-2019} for  transonic flows, \cite{benci-2000} for quantum physics and  \cite{cherfils-2005} for reaction-diffusion systems.

Equations \eqref{eq:main} with $p\not=q$ are also referred to as the equations with the $(p,q)$-growth because of the gap between the coercivity and growth conditions: if $p\leq q$ and $0\leq a(x)\leq L$, then for every $\xi\in \mathbb{R}^{N}$

\[
|\xi|^{p}\leq (|\xi|^{p-2}+a(x)|\xi|^{q-2})|\xi|^2\leq C(1+|\xi|^q),\quad C=const>0.
\]
These equations fall into the class of equations with nonstandard
growth conditions which have been actively studied during the last
decades in the cases of constant or variable exponents $p$ and
$q$. We refer to the works \cite{Alves-Radulescu-2020,
chlebicka-2018, colombo2015bounded, colombo2015regularity, esposito2004sharp,
Gasinski-Winkert-2020,Hasto-Ok-2019,Liu-Dai-2018,marcellini-1991,
ok-2020,Radulescu-2019, Radulescu-Zhang-2018} and references therein for the results on the existence and regularity of solutions, including optimal regularity results \cite{esposito2004sharp}.

Results on the existence of solutions to the evolution double phase equations can be found in papers \cite{BDM-2013,Singer-2015,Singer-2016}. These works deal with the Dirichlet problem for systems of parabolic equations of the form

\begin{equation}\label{bas:model}
   u_t-\operatorname{div}a(x,t,\nabla u)=0,
\end{equation}
where the flux $a(x,t,\nabla u)$ is assumed to satisfy the $(p,q)$-growth  conditions and certain regularity assumptions. As a partial case, the class of equations \eqref{bas:model} includes equation \eqref{eq:main} with constant exponents $p\leq q$ and a nonnegative bounded coefficient $a(x,t)$. It is shown in \cite[Th.1.6]{BDM-2013} that if

\[
2\leq p \leq q <p+\frac{4}{N+2},
\]
then problem \eqref{eq:main} with $F \equiv 0$ has a \textbf{very weak solution}

\[
\text{
$u\in L^p(0,T;W^{1,p}_0(\Omega))\cap L^q_{loc}(0,T;W^{1,q}_{loc}(\Omega))$ \quad with \quad $u_t\in L^{\frac{p}{q-1}}(0,T;W^{-1,\frac{p}{q-1}}(\Omega))$},
\]
provided that $u_0\in W^{1,r}_0(\Omega)$, $r=\frac{p(q-1)}{p-1}$. Moreover, $|\nabla u|$ is bounded on every strictly interior cylinder $Q'_{T}\Subset Q_T$ separated away from the parabolic boundary of $Q_T$. In \cite{Singer-2015} these results were extended to the case

\[
\dfrac{2N}{N+2}<p<2,\qquad p\leq q<p+\dfrac{4}{N+2}.
\]
Paper \cite{Singer-2016} deals with \textbf{weak solutions} of systems of equations of the type \eqref{bas:model} with $(p,q)$ growth conditions. When applied to problem \eqref{eq:main} with constant $p$, $q$, $b\equiv 0$ and {$a(\cdot,t)\in C^{\alpha}(\Omega)$} with some $\alpha\in (0,1)$ for a.e. $t\in (0,T)$, the result of \cite{Singer-2016} guarantees the existence of a weak solution
\[
u\in L^{p}(0,T;W^{1,p}_0(\Omega))\cap L_{loc}^{q}(0,T;W_{loc}^{1,q}(\Omega))\cap L^{\infty}(0,T;L^{2}(\Omega)),
\]
provided that the exponents $p$ and $q$ obey the inequalities
\[
\dfrac{2N}{N+2}<p<q<p+\dfrac{\alpha\min\{2,p\}}{N+2}.
\]
The proofs of the existence theorems in \cite{BDM-2013,Singer-2015,Singer-2016} rely on the property of local higher integrability of the gradient, $|\nabla u|^{p+\delta}\in L^{1}(Q'_T)$ for every sub-cylinder $Q'_T\Subset Q_T$. The maximal possible value of $\delta>0$ indicates the admissible gap between the exponents $p$ and $q$ and vary in dependence on the type of the solution.

Equation \eqref{eq:main} with constant exponents $p$ and $q$
furnishes a prototype of  the equations recently studied in papers
\cite{bogelein-2013,de-2020,giannetti-2020,marcellini-2020} in the
context of {weak or variational solutions.} The
proofs of existence also use the local higher integrability of the
gradient, but for the existence of variational solutions a weaker
assumption on the gap $q-p$ is required.

Nonhomogeneous parabolic equations of the form  \eqref{bas:model} with the flux $a(x,t,\nabla u)$ controlled by a generalized $N$-function were studied in \cite{Chlebicka-2019} in the context of Musielak-Orlicz spaces. The class of equations studied in \cite{Chlebicka-2019} includes, as a partial case, equation \eqref{eq:main} with $b\equiv 0$ and variable exponents $p$, $q$. It is shown that problem \eqref{eq:main} with $b=0$ and bounded data $u_0$ and $f_0$ admits a unique solution $u\in L^{\infty}(0,T;L^{2}(\Omega))\cap L^{1}(0,T;W^{1,1}_0(\Omega))$ with $\nabla u\in L_M(Q_T;\mathbb{R}^N)$, where $L_M$ denotes the Musielak-Orlicz space defined by the flux $a(x,t,\xi)$. We refer here to \cite{Chlebicka-2019,Fan-2012}, the survey article \cite{chlebicka-2018} and references therein for the issues of solvability of elliptic and parabolic equations in the Musielak-Orlicz spaces.

\subsection{Description of results} In the present work, we prove the existence of \textbf{strong solutions} of problem \eqref{eq:main}. By the strong solution we mean a solution whose time derivative is not a distribution but an element of a Lebesgue space, and the flux has better integrability properties than the properties prompted by the energy equality (the rigorous formulation is given in Definition \ref{def:weak}).
We consider first the case $b\equiv 0$.
We show that if the exponents {$p$, $q$ and the data $a$, $u_0$, $f_0$} are sufficiently smooth, and if the gap between the exponents satisfies the condition

\begin{equation}
\label{eq:p-q}
\text{$\dfrac{2N}{N+2}<p(z)\leq q(z)<p(z)+ { \dfrac{2}{(N+2)+\frac{2N}{p^-}}}$ in $Q_T$},\qquad p^-=\inf_{Q_T}p(z),
\end{equation}
then problem \eqref{eq:main} has a strong solution {$u$} with

\begin{equation}
\label{eq:F}
u_t\in L^{2}(Q_T), \qquad \operatorname{ess}\sup_{(0,T)}\mathcal{F}(u(\cdot,t),t)<\infty,\quad \mathcal{F}(u(\cdot,t),t)\equiv \int_{\Omega}(|\nabla u|^{p(z)}+a(z)|\nabla u|^{q(z)})\,dx.
\end{equation}
Moreover, the solution possesses the property of global higher integrability of the gradient:

\begin{equation}
\label{eq:osc-0}
\int_{Q_T}|\nabla u|^{p(z)+r}\,dz\leq C\quad \text{for every $0<r<\frac{4}{(N+2)+\frac{2N}{p^-}}$}
\end{equation}
with a finite constant $C$ depending only on $\mathcal{F}(u_0,0)$, $N$, $r$, and the properties of {$p(\cdot)$ and $q(\cdot)$}. The same existence result is valid for problem \eqref{eq:main} with the regularized nondegenerate flux

\[
\left((\epsilon^2+|\nabla u|^2)^{\frac{p(z)-2}{2}}+a(z)(\epsilon^2+|\nabla u|^2)^{\frac{q(z)-2}{2}}\right)\nabla u,\qquad \epsilon>0.
\]

In the case $b\not\equiv 0$, the existence of strong solutions of
problem \eqref{eq:main}  is proven under the additional assumption
\[
2\leq \inf_{Q_T}\sigma(z)\leq \sup_{Q_T}\sigma(z)<1+\dfrac{1}{2}\inf_{Q_T}p(z),
\]
which restricts the considerations to the degenerate case
$p(z)>2$ in $Q_T$.  For the uniqueness of strong solutions we assume that either $b(z)\leq 0$, or {$\sigma \equiv 2$  in $Q_T$}.

Property \eqref{eq:osc-0} of global higher integrability of the
gradient turns out to be crucial for the proof of existence of
strong solutions.  Unlike the traditional approach based on the
use of Caccioppoli type inequalities and scaling,  estimate
\eqref{eq:osc-0} is proved by means of an interpolation inequality
of the Gagliardo-Nirenberg type.  This method enables one to
derive global estimates, although at expense of a stronger
restriction on the gap between {$p(z)$ and $q(z)$}.

Inequality \eqref{eq:osc-0} allows one to extend the results to the multiphase equations
\[
u_t-\operatorname{div}\left(|\nabla u|^{p(z)-2}\nabla u+\sum_{i=1}^Ka_i(z)|\nabla u|^{q_i(z)-2}\nabla u\right)=F(z,u),
\]
provided each of {$q_i(\cdot)$} satisfies \eqref{eq:p-q}.

\section{The function spaces}
We begin with a brief description of the Lebesgue and
Sobolev spaces with variable exponents. A detailed insight into
the theory of these spaces and a review of the bibliography can be
found in \cite{UF,DHHR-2011,KR}. Let $\Omega \subset \mathbb{R}^N$ be a
bounded domain with Lipschitz continuous boundary $\partial
\Omega$. Define the set

$$
\mathcal{P}(\Omega):= \{\text{measurable
functions on} \ \Omega \ \text{with values in}\ (1, \infty)\}.$$

\subsection{Variable Lebesgue spaces} {Throughout the rest of the paper we assume that $\Omega\subset \mathbb{R}^{N}$, $N\geq 2$, is a bounded domain with the boundary $\partial\Omega\in C^2$.} Given $r \in \mathcal{P}(\Omega)$, we introduce the modular

\begin{equation}
\label{eq:modular} A_{r(\cdot)}(f)= \int_{\Omega} |f(x)|^{r(x)}
~dx
\end{equation}
and the set

\[
L^{r(\cdot)}(\Omega) = \{f: \Omega \to \mathbb{R}\ |\ \text{measurable on}\ \Omega,\ A_{r(\cdot)}(f) < \infty\}.
\]
The set $L^{r(\cdot)}(\Omega)$ equipped with the Luxemburg norm

$$
\|f\|_{r(\cdot), \Omega}= \inf \left\{\lambda>0 : A_{r(\cdot)}\left(\frac{f}{\lambda}\right) \leq 1\right\}
$$
becomes a Banach space. By convention, from now on we use the notation

$$
r^- := \ess\min_{x \in \Omega} r(x), \quad r^+ := \ess\max_{x \in \Omega} r(x).
$$
If $r \in \mathcal{P}(\Omega)$ and $1< r^- \leq r(x) \leq r^+ < \infty$ in $\Omega$, then the following properties hold.
\begin{enumerate}[label=(\roman*)]
\item $L^{r(\cdot)}(\Omega)$ is a reflexive and separable Banach space.
\item For every $f \in L^{r(\cdot)}(\Omega)$

\begin{equation}
\label{eq:mod-2}
\min\{\|f\|^{r^-}_{r(\cdot), \Omega}, \|f\|^{r^+}_{r(\cdot), \Omega}\} \leq A_{r(\cdot)}(f) \leq \max\{\|f\|^{r^-}_{r(\cdot), \Omega}, \|f\|^{r^+}_{r(\cdot), \Omega}\}.
\end{equation}
\item For every $f \in L^{r(\cdot)}(\Omega)$ and $g \in L^{r'(\cdot)}(\Omega)$, the generalized H\"older inequality holds:

\begin{equation}
\label{eq:Holder}
\int_{\Omega} |fg| \leq \left(\frac{1}{r^-} + \frac{1}{(r')^-} \right) \|f\|_{r(\cdot), \Omega} \|g\|_{r'(\cdot), \Omega} \leq 2 \|f\|_{r(\cdot), \Omega} \|g\|_{r'(\cdot), \Omega},
\end{equation}
where $r'= \frac{r}{r-1}$ is the conjugate exponent of $r$.
\item If $p_1, p_2 \in \mathcal{P}(\Omega)$ and satisfy the inequality $p_1(x) \leq p_2(x)$ a.e. in $\Omega$, then $L^{p_1(\cdot)}(\Omega)$ is continuously embedded in $L^{p_2(\cdot)}(\Omega)$ and for all $u \in L^{p_2(\cdot)}(\Omega)$
 \begin{equation}
 \label{eq:emb-1}\|u\|_{p_1(\cdot), \Omega} \leq C\|u\|_{p_2(\cdot), \Omega},\qquad C=C(|\Omega|, p_1^\pm, p_2^\pm).
 \end{equation}
 \item For every sequence $\{f_k\} \subset L^{r(\cdot)}(\Omega)$ and $f \in L^{r(\cdot)}(\Omega)$
\begin{equation}
\label{eq:conv-modular}
\|f_k-f\|_{r(\cdot),\Omega} \to 0 \quad \text{iff $ A_{r(\cdot)}(f_k-f) \to 0$ as $k \to \infty$}.
\end{equation}
\end{enumerate}

\subsection{Variable Sobolev spaces}
The variable Sobolev space $W^{1,r(\cdot)}_0(\Omega)$ is the set of functions

\[
W^{1,r(\cdot)}_0(\Omega)= \{u: \Omega \to \mathbb{R}\ |\  u \in L^{r(\cdot)}(\Omega)\cap W^{1,1}_0(\Omega),\; |\nabla u| \in L^{r(\cdot)}(\Omega) \}
\]
equipped with the norm
$$\|u\|_{W^{1,r(\cdot)}_0(\Omega)}= \|u\|_{r(\cdot), \Omega} + \|\nabla u\|_{r(\cdot), \Omega}.$$

\noindent If $r \in C^0(\overline{\Omega})$, the Poincar\'e inequality holds: for every $u\in W_0^{1,r(\cdot)}(\Omega)$

\begin{equation}
\label{eq:Poincare}
\|u\|_{r(\cdot),\Omega} \leq C \|\nabla u\|_{r(\cdot),\Omega}.
\end{equation}
Inequality \eqref{eq:Poincare} means that the equivalent norm of $W^{1,r(\cdot)}_0(\Omega)$ is given by

\begin{equation}
\label{eq:equiv-norm}
\|u\|_{W^{1,r(\cdot)}_0(\Omega)}=\|\nabla u\|_{r(\cdot),\Omega}.
\end{equation}
Let us denote by $C_{{\rm log}}(\overline{\Omega})$ the subset of $\mathcal{P}(\Omega)$ composed of the functions continuous on $\overline{\Omega}$ with the logarithmic modulus of continuity:

\[
{p} \in C_{{\rm log}}(\overline{\Omega})\quad \Leftrightarrow\quad |p(x)-p(y)|\leq \omega(|x-y|)\quad \forall x,y\in \overline{\Omega}, \;|x-y|<\frac{1}{2},
\]
where {$\omega$} is a nonnegative function such that

\[
\limsup_{s\to 0^+}\omega(s)\ln \frac{1}{s}=C,\quad C=const.
\]
If {$r \in C_{{\rm log}}(\overline{\Omega})$}, then the set $C_{c}^\infty(\Omega)$ of smooth functions with finite support is dense in $W^{1,r(\cdot)}_0(\Omega)$. This property allows one to use the equivalent definition of the space $W^{1,r(\cdot)}_0(\Omega)$:

\[
W^{1,r(\cdot)}_0(\Omega)=\left\{\text{the closure of $C_{c}^{\infty}(\Omega)$ with respect to the norm $\|\cdot\|_{W^{1,r(\cdot)}_0(\Omega)}$}\right\}.
\]
Given a function $u\in W^{1,r(\cdot)}_0(\Omega)$ with {$r \in C_{{\rm log}}(\overline{\Omega})$}, the smooth approximations of $u$ in $W^{1,r(\cdot)}_0(\Omega)$ can be obtained by means of the Friedrichs mollifiers.

The following analogue of the Sobolev embedding theorem holds. Given {$p \in C^0(\overline{\Omega})$}, $1<p^-\leq p^+<\infty$, let us introduce the Sobolev conjugate exponent

\[
p^\ast(x)=\begin{cases}
\dfrac{Np(x)}{N-p(x)} & \text{if $p(x)<N$},
\\
\text{any number from $[1,\infty)$} & \text{if $p(x)\geq N$}.
\end{cases}
\]
If {$q \in C^{0}(\overline{\Omega})$} and $\inf_{\Omega}(p^\ast(x)-q(x))>0$, then for every $u\in W^{1,p(\cdot)}_0(\Omega)$
\begin{equation}
\label{eq:embed-var}
\|u\|_{q(\cdot),\Omega}\leq C\|u\|_{W^{1,p(\cdot)}_0(\Omega)},\quad C=C(p^\pm,q^\pm,|\Omega|,N),
\end{equation}
and the embedding
$W^{1,p(\cdot)}_0(\Omega)\subset L^{q(\cdot)}(\Omega)$ is compact.

The dual space to $W=W^{1,q(\cdot)}_0(\Omega)$ is the set of linear bounded functionals over $W$:
$W'={W^{-1,q'(\cdot)}(\Omega)}$. $W'$ consists of the vectors
$G=(g_0,g_i,\ldots,g_N)$, $g_i\in L^{q'(\cdot)}(\Omega)$, such
that for every $u\in W$
\[
\langle
G,u\rangle_{W',W}=\int_{\Omega}\left(g_0u +\sum_{i=1}^{N}g_iD_{x_i}u\right)\,dx.
\]
Since the equivalent norm of $W$ is given by \eqref{eq:equiv-norm},  for every $G\in W'$ there exists a function $F=(F_1,\ldots,F_N)$ such that $F_i\in L^{q'(\cdot)}(\Omega)$ and for every $u\in W$
\[
\langle G,u\rangle_{W',W}=\sum_{i=1}^N
\int_{\Omega}F_iD_{x_i}u\,dx.
\]

\subsection{Spaces of functions depending on $x$ and $t$}For the study of parabolic problem \eqref{eq:main} we need the spaces of functions depending on $z=(x,t)\in Q_T$. {Given a function $q \in C_{{\rm log}}(\overline{Q}_T)$,} we introduce the spaces
\begin{equation}
\label{eq:spaces}
\begin{split}
 & \mathcal{V}_{q(\cdot,t)}(\Omega) = \{u: \Omega \to \mathbb{R}\ |\  u \in L^2(\Omega)
\cap W_0^{1,1}(\Omega),\,|\nabla u|^{q(x,t)}\in L^{1}(\Omega)
\},\quad t\in (0,T),
\\
& \mathcal{W}_{q(\cdot)}(Q_T)= \{u : (0,T) \to \mathcal{V}_{q(\cdot,t)}(\Omega) \ |\ u \in L^2(Q_T), |\nabla u|^{q(z)}\in L^1(Q_T)\}.
\end{split}
\end{equation}
The norm of {$\mathcal{W}_{q(\cdot)}(Q_T)$} is defined by
\[
\|u\|_{\mathcal{W}_{{q(\cdot)}}(Q_T)}=\|u\|_{2,Q_T}+\|\nabla u\|_{q(\cdot),Q_T}.
\]
Since ${q} \in C_{{\rm log}}(\overline{Q}_T)$, the space {$\mathcal{W}_{q(\cdot)}(Q_T)$} is the closure of $C_{c}^{\infty}(Q_T)$ with respect to this norm.

\subsection{Musielak-Sobolev spaces}
Let $a_0:\Omega\mapsto [0,\infty)$ be a given function, $a_0\in C^{0,1}(\overline{\Omega})$. Assume that the exponents $p(x), q(x)\in C^{0,1}(\overline{\Omega})$ take values in the intervals $(p^-,p^+)$, $(q^-,q^+)$, and $p(x)\leq q(x)$ in $\Omega$.
Set
\[
r(x)=\max\{2,p(x)\},\quad s(x)=\max\{2,q(x)\}
\]
and consider the function
\[
\mathcal{H}(x,t)=t^{r(x)}+a_0(x)t^{s(x)},\quad t\geq 0,\quad x\in \Omega.
\]
The set
\[
L^\mathcal{H}(\Omega)=\left\{u:\Omega\mapsto \mathbb{R}\,|\, \text{$u$ is measurable},\,\rho_{\mathcal{H}}(u)=\int_{\Omega}\mathcal{H}(x,|u|)\,dx<\infty \right\}
\]
equipped with the Luxemburg norm
\[
\|u\|_{\mathcal{H}}=\inf \left\{\lambda>0:\,\rho_{\mathcal{H}}\left(\dfrac{u}{\lambda}\right)\leq 1\right\}
\]
becomes a Banach space. The space $L^\mathcal{H}(\Omega)$ is separable and reflexive \cite{Fan-2012}. By $\mathcal{V}(\Omega)$ we denote the Musielak-Sobolev space
\[
\mathcal{V}(\Omega)=\left\{u\in L^{\mathcal{H}}(\Omega)\,:\,|\nabla u|\in L^{\mathcal{H}}(\Omega)\right\}
\]
with the norm
\[
\|u\|_{\mathcal{V}}=\|u\|_{\mathcal{H}}+\|\nabla u\|_{\mathcal{H}}.
\]
The space $\mathcal{V}_0(\Omega)$ is defined as the closure of $C_{c}^\infty(\Omega)$ with respect to the norm of $\mathcal{V}(\Omega)$.

\subsection{Dense sets in $W^{1,p(\cdot)}_0(\Omega)$ and $\mathcal{V}_0(\Omega)$}
Let $\{\phi_i\}$ and $\{\lambda_i\}$ be the eigenfunctions and the corresponding eigenvalues of the Dirichlet problem for the Laplacian:
\begin{equation}
\label{eq:eigen}
(\nabla \phi_i,\nabla \psi)_{2,\Omega}=\lambda_i(\phi_i,\psi)\qquad \forall \psi\in H^{1}_0(\Omega).
\end{equation}
The functions $\phi_i$ form an orthonormal basis of $L^2(\Omega)$ and are mutually orthogonal in $H^1_0(\Omega)$. If $\partial\Omega\in C^k$, $k\geq 1$, then $\phi_i\in C^{\infty}(\Omega)\cap H^{k}(\Omega)$. Let us denote by $H^k_{\mathcal{D}}(\Omega)$ the subspace of the Hilbert space $H^{k}(\Omega)$ composed of the functions $f$ for which
\[
f=0,\quad \Delta f=0,\quad \ldots, \Delta^{[\frac{k-1}{2}]} f=0\quad \text{on $\partial\Omega$},\qquad H^{0}_{\mathcal{D}}(\Omega)=L^2(\Omega).
\]
The relations
\[
[f,g]_{k}=\begin{cases}
(\Delta^{\frac{k}{2}}f,\Delta^{\frac{k}{2}}g)_{2,\Omega} & \text{if $k$ is even},
\\
(\Delta^{\frac{k-1}{2}}f,\Delta^{\frac{k-1}{2}}g)_{H^1(\Omega)} & \text{if $k$ is odd}
\end{cases}
\]
define an equivalent inner product on ${H}^{k}_{\mathcal{D}}(\Omega)$:
$
[f,g]_k=\displaystyle\sum_{i=1}^\infty \lambda_i^k f_ig_i$,
where $f_i$, $g_i$ are the Fourier coefficients of $f$, $g$ in the basis $\{\phi_i\}$ of $L^2(\Omega)$. The corresponding equivalent norm of $H^k_{\mathcal{D}(\Omega)}$ is defined by $\|f\|^2_{{H}^{k}_{\mathcal{D}}(\Omega)}=[f,f]_k$. Let $f^{(m)}=\sum_{i=1}^{m}f_i\phi_i$ be the partial sums of the Fourier series of $f\in L^{2}(\Omega)$. The following assertion is well-known.
\begin{pro}
\label{pro:series}
Let $\partial\Omega\in C^k$, $k\geq 1$. A function $f$ can be represented by the Fourier series in the system $\{\phi_i\}$,
convergent in the norm of $H^k(\Omega)$, if and only if $f\in H^{k}_{\mathcal{D}}(\Omega)$. If $f\in H^{k}_{\mathcal{D}}(\Omega)$, then the series $\sum_{i=1}^{\infty}\lambda_i^kf_i^2$ is convergent, its sum is bounded by $C\|f\|_{H^k(\Omega)}$ with an independent of $f$ constant $C$, and $\|f^{(m)}-f\|_{H^k(\Omega)}\to 0$ as $m\to \infty$. If $k\geq [\frac{N}{2}]+1$, then the Fourier series in the system $\{\phi_i\}$ of every function $f\in H^{k}_{\mathcal{D}}(\Omega)$ converges to $f$ in $C^{k-[\frac{N}{2}]-1}(\overline{\Omega})$.
\end{pro}
\begin{pro}[\cite{DNR-2012}, Th.~4.7, Proposition 4.10]
\label{pro:density}
Let $\partial \Omega\in Lip$ and $p(x)\in C_{\rm log}(\overline{\Omega})$. Then the set $C^\infty_{c}(\Omega)$ is dense in $W_0^{1,p(\cdot)}(\Omega)$.
\end{pro}
Let us denote $\mathcal{P}_m=\operatorname{span}\{\phi_1,\ldots,\phi_m\}$, where $\phi_i$ are the solutions of problem \eqref{eq:eigen}.
\begin{Lem}
\label{le:dense-W}
If $\partial\Omega\in C^k$ with
\begin{equation}
\label{eq:k-1}
k\geq N\left(\frac{1}{2}+\frac{1}{N}-\frac{1}{p^+}\right),\qquad p^+=\max_{\overline{\Omega}}p(x),
\end{equation}
and $p\in C_{{\rm log}}(\overline{\Omega})$, then $\bigcup_{m=1}^{\infty}\mathcal{P}_m$ is dense in $W^{1,p(\cdot)}_0(\Omega)$.
\end{Lem}
\begin{proof}
Given $v\in W^{1,p(\cdot)}_0(\Omega)$ we have to show that for every $\epsilon>0$ there is $m\in \mathbb{N}$ and $v_m\in \mathcal{P}_m$ such that $\|v-v_m\|_{W^{1,p(\cdot)}_0(\Omega)}<\epsilon$. Fix some $\epsilon>0$. By Proposition \ref{pro:density} there is $v_{\epsilon}\in C_{c}^{\infty}(\Omega)\subset H^k_{\mathcal{D}}(\Omega)$ such that $\|v-v_\epsilon\|_{W^{1,p(\cdot)}_0(\Omega)}<\epsilon/2$. By Proposition \ref{pro:series} $v_\epsilon^{(m)}(x)=\sum_{i=1}^mv_i\phi_i(x)\in H^{k}(\Omega)$ and $v_{\epsilon}^{(m)}\to v_\epsilon$ in $H^k(\Omega)$, therefore for every $\delta>0$ there is $m\in \mathbb{N}$ such that $\|v_\epsilon-v_\epsilon^{(m)}\|_{H^k(\Omega)}<\delta$. Since $k,N,q$ satisfy condition \eqref{eq:k-1}, the embeddings $H_{\mathcal{D}}^k(\Omega) \subset W_0^{1, q^+}(\Omega) \subseteq W^{1,q(\cdot)}_0(\Omega)$ are continuous:
\[
\|w\|_{W^{1,q(\cdot)}_0(\Omega)}\leq
C\|w\|_{W_0^{1, q^+}(\Omega)}\leq  C'\|w\|_{H^k(\Omega)}\qquad \forall w\in H_{\mathcal{D}}^k(\Omega)
\]
with independent of $w$ constants $C$, $C'$. Set $C'\delta=\epsilon/2$. Then
\[
\|v_\epsilon-v_{\epsilon}^{(m)}\|_{W_0^{1, p(\cdot)}(\Omega)} \leq C'\|v_\epsilon-v_{\epsilon}^{(m)}\|_{H^{k}(\Omega)} <C' \delta =\frac{\epsilon}{2}.
\]
It follows that
\[
\|v-v_\epsilon^{(m)}\|_{W^{1,p(\cdot)}_0(\Omega)} \leq \|v-v_\epsilon\|_{W^{1,p(\cdot)}_0(\Omega)} +\|v_\epsilon-v_\epsilon^{(m)}\|_{W^{1,p(\cdot)}_0(\Omega)}<\frac{\epsilon}{2}+ C'\delta =\epsilon.
\]
\end{proof}
\begin{cor}
\label{cor:density-par}
If $p\in C_{{\rm log}}(\overline{Q}_T)$ and condition \eqref{eq:k-1} is fulfilled, then
\[
\text{$\displaystyle \left\{v(x,t):\,v=\sum_{i=1}^{\infty}v_i(t)\phi_i(x), \,v_i(t)\in C^{0,1}[0,T]\right\}$ is dense in $\mathcal{W}_{p(\cdot)}(Q_T)$}.
\]
\end{cor}
\begin{pro}
\label{pro:density-2}
Let $\partial\Omega \in Lip$ and $a_0,p,q\in C^{0,1}(\overline{\Omega})$. If $p(x)\leq q(x)$ in $\Omega$ and
\[
\frac{s^+}{r^-}\leq 1+\frac{1}{N},\quad s^+=\max_{\overline{\Omega}}\{\max\{2,q(x)\}\},\quad r^-=\min_{\overline{\Omega}}\{\max\{2,p(x)\}\},
\]
then $C^\infty_c(\Omega)\cap \mathcal{V}(\Omega)$ is dense in $\mathcal{V}_0(\Omega)$.
\end{pro}
The assertion of Proposition \ref{pro:density-2} follows from \cite[Th.3.1]{Chlebicka-2019-1} or \cite[Th.6.4.7]{Hasto-Harjulehto-2019-book}. A straightforward checking of all conditions listed in \cite{Hasto-Harjulehto-2019-book} is given in \cite[Theorem 2.21]{Blanco-2021}.
\begin{Lem}
\label{le:density-3}
If $a_0,p,q\in C^{0,1}(\overline{\Omega})$ and $\partial \Omega\in C^{k}$ with
\begin{equation}
\label{eq:k-2}
k\geq N\left(\frac{1}{2}+\frac{1}{N}-\frac{1}{s^+}\right),
\end{equation}
then $\bigcup_{m=1}^{\infty}\mathcal{P}_m$ is dense in $\mathcal{V}_0(\Omega)$.
\end{Lem}
We omit the detailed proof which is an imitation of the proof of Lemma \ref{le:dense-W}: by Proposition \ref{pro:density-2} $C_{c}^\infty(\Omega)$ is dense in $\mathcal{V}_0(\Omega)$, and since $C_{c}^\infty(\Omega)\subset H^{(k)}_{\mathcal{D}}(\Omega)$ every $v_\epsilon\in H^k_{\mathcal{D}}(\Omega)$ can be approximated by $v_{\epsilon}^{(m)}\in \mathcal{P}_m$.

\section{Assumptions and main results}
Let $p,q: Q_T \mapsto \mathbb{R}$  be measurable functions satisfying the conditions
\begin{equation}\label{assum1}
\begin{split}
&\frac{2N}{N+2} < p_- \leq p(z) \leq p_+ \ \text{in} \ \overline{Q}_T,\\
&\frac{2N}{N+2} < q_- \leq q(z) \leq q_+ \ \text{in} \ \overline{Q}_T, \qquad p^\pm,\,q^\pm=const.
\end{split}
\end{equation}
Moreover, let us assume that $p,\,q\in W^{1,\infty}(Q_T)$ as functions of variables $z=(x,t)$: there exist positive constants $C^\ast$, $C^{\ast\ast}$, $C_\ast$, $C_{\ast\ast}$ such that

\begin{equation}
\label{eq:Lip-p-q}
\begin{split}
&\ess\sup_{Q_T}|\nabla p| \leq C_\ast < \infty,\quad \ess\sup_{Q_T}|
p_t| \leq C^\ast,\\
&\ess\sup_{Q_T}|\nabla q| \leq C_{\ast\ast} < \infty,\quad \ess\sup_{Q_T}|
q_t| \leq C^{\ast\ast}.
\end{split}
\end{equation}
The modulating coefficient {$a(\cdot)$} is assumed to satisfy the following conditions:
\begin{equation}
\label{eq:a}
{\text{$a(z) \geq 0$ in $\overline{Q}_T$},\qquad {a} \in C([0,T]; W^{1,\infty}(\Omega)), \qquad \ess\sup_{Q_T}|a_t| \leq C_a, \quad C_a=const.}
\end{equation}
We do not impose any condition on the null set of the function {$a$} in $\overline{Q}_T$ and do not distinguish between the cases of degenerate and singular equations. It is possible that $p(z)<2$ and $q(z)>2$ at the same point $z\in Q_T$.

\begin{define}
\label{def:weak}
A function {$u:Q_T\mapsto \mathbb{R}$} is called {\bf strong solution} of problem \eqref{eq:main} if

\begin{enumerate}

\item {$u \in  \mathcal{W}_{q(\cdot)}(Q_T)$,
    $u_t \in L^2(Q_T)$, \text{$|\nabla u| \in L^{\infty}(0,T; L^{s(\cdot)}(\Omega))$ with $s(z)= \max\{2,p(z)\}$},
\item for every $\psi \in \mathcal{W}_{q(\cdot)}(Q_T) $ with $\psi_t \in L^2(Q_T)$}
\begin{equation}
\label{eq:def}
\int_{Q_T} u_t \psi ~dz + \int_{Q_T} (|\nabla
u|^{p(z)-2} + a(z) |\nabla u|^{q(z)-2}) \nabla u \cdot \nabla \psi ~dz =
\int_{Q_T} F(z,u) \psi \,dz,
\end{equation}
\item for every $\phi \in C_0^1(\Omega)$
\[
\int_{\Omega} (u(x,t)-u_0(x)) \phi  ~dx \to 0\quad\text{as $t \to
0$}.
\]
\end{enumerate}
\end{define}

The main results are given in the following theorems.

\begin{thm}\label{PPresu2}
Let $\Omega\subset \mathbb{R}^N$, $N\geq 2$, be a bounded domain with the boundary $\partial\Omega\in C^k$, $k\geq 2+[\frac{N}{2}]$. Assume that {$p(\cdot)$, $q(\cdot)$} satisfy conditions \eqref{assum1}, \eqref{eq:Lip-p-q}, and there exists a constant

\[
\text{$r\in \displaystyle \left(0,r^\ast\right)$,\quad  $r^\ast=\dfrac{4p^-}{p^-(N+2)+ 2N}$},
\]
such that

\begin{equation}
\label{eq:oscillation}
\text{$p(z) \leq q(z) \leq p(z) + { \frac{r}{2}}$ in {$\overline{Q}_T$}}.
\end{equation}
If {$a(\cdot)$} satisfies conditions \eqref{eq:a} and {$b \equiv 0$}, then for every
$f_0 \in L^2(0,T; W_0^{1,2}(\Omega))$ and $u_0 \in
W_0^{1,2}(\Omega)$ with

\begin{equation}
\label{eq:ini}
{\int_{\Omega}\left(|\nabla u_0|^{2}+|\nabla u_0|^{p(x,0)} + a(x,0) |\nabla u_0|^{q(x,0)}\right)\,dx=K<\infty}
\end{equation}
problem \eqref{eq:main} has a unique strong solution {$u$}. This solution satisfies the estimate

\begin{equation}
\label{eq:strong-est}
\|u_t\|_{2,Q_T}^{2} +
\operatorname{ess}\sup_{(0,T)}\int_{\Omega} \left(|\nabla
u|^{s(z)}+a(z)|\nabla u|^{q(z)}\right)\,dx +\int_{Q_T} |\nabla
u|^{p(z)+r}\,dz \leq C
\end{equation}
with the exponent $s(z)=\max\{2,p(z)\}$ and a constant $C$ which
depends on $N, \partial \Omega$, $T, p^\pm, q^\pm$, $r$, the
constants in conditions \eqref{eq:Lip-p-q}, \eqref{eq:a},
$\|f_0\|_{L^{2}(0,T;W^{1,2}_0(\Omega))}$ and $K$.
\end{thm}

\begin{thm}\label{th:exist-source}
Let in the conditions of Theorem \ref{PPresu2}, $b \not\equiv 0$.
\begin{itemize}
\item[{\rm (i)}] Assume that {$b,\sigma$} are measurable bounded functions defined on $Q_T$,
\begin{equation}
\label{eq:p-q-source}
\begin{split}
&
\|\nabla b\|_{\infty,Q_T}<\infty, \quad \|\nabla \sigma\|_{\infty,Q_T}<\infty,
\\
& 2\leq \sigma^-\leq \sigma^+<1+\dfrac{p^-}{2},\quad \sigma^-=\operatorname{ess}\inf_{Q_T}\sigma(z),\quad \sigma^+=\operatorname{ess}\sup_{Q_T}\sigma(z).
\end{split}
\end{equation}
Then for every
$f_0 \in L^2(0,T; W_0^{1,2}(\Omega))$ and $u_0 \in
W_0^{1,2}(\Omega)$ satisfying condition \eqref{eq:ini}
problem \eqref{eq:main} has at least one strong solution {$u$}. The solution {$u$} satisfies estimate \eqref{eq:strong-est} with the constant depending on the same quantities as in the case $b\equiv 0$ and on $\|\nabla b\|_{\infty,Q_T}$, $\|\nabla \sigma\|_{\infty,Q_T}$, $\sigma^\pm$, $\operatorname{ess}\sup_{Q_T}|b|$.

\item[{\rm (ii)}] The strong solution is unique if {$p(\cdot),q(\cdot)$} satisfy the conditions of Theorem \ref{PPresu2} and either {$\sigma \equiv 2$, or $b(z)\leq 0$ in $Q_T$}.
\end{itemize}
\end{thm}

An outline of the work. In Section \ref{sec:prelim} we collect several auxiliary assertions. We present estimates on the gradient trace on $\partial\Omega$ for the functions from variable Sobolev spaces and formulate the interpolation inequality which enables us to prove global higher integrability of the gradient. This property turns out to be the key element in the proof of the existence theorems for problem \eqref{eq:main} and the regularized problem \eqref{eq:reg-prob}.

A solution of problem \eqref{eq:main} is obtained as the limit of the family of solutions of the nondegenerate problems \eqref{eq:reg-prob} with the regularized fluxes

\[
\left((\epsilon^2+|\nabla u|^2)^{\frac{p(z)-2}{2}}+a(z)(\epsilon^2+|\nabla u|^2)^{\frac{q(z)-2}{2}}\right)\nabla u,\qquad \epsilon>0.
\]

For every $\epsilon\in (0,1)$ problem
\eqref{eq:reg-prob} is solved with the method of Galerkin. In
Section \ref{sec:reg-problem} we formulate the problems for the
approximations.
%and construct a sequence of finite-dimensional
%approximations for the initial function $u_0$ in the same basis we use to approximate the
%solution of the parabolic problem. In the nondegenerate case, $q(x,0)\leq 2$ in $\Omega$, this sequence is
%obtained in a standard way, while in the case $\sup_\Omega q(x,0)>2$ the choice of the sequence becomes an independent problem. We construct it as a sequence of finite-dimensional approximations of
%the solution of the degenerate double phase elliptic equation (see \eqref{eq:ell-main}) with variable exponents $r(x)=\max\{2,p(x,0)\}$ and $s(x)=\max\{2,q(x,0)\}$, and the right-hand side depending on $u_0$. This problem is solved with the method of Galerkin in the framework of Musielak-Orlicz spaces.

Section \ref{sec:a-priori} is devoted to derive a priori estimates
on the approximate solutions and their derivatives. For the
convenience of presentation, we separate the cases when $b\equiv
0$ and the source function is independent of the solution, and
$b\not\equiv 0$. Since no restriction on the sign of $b$ is
imposed, in the latter case derivation of the a priori estimates
requires additional restrictions on the range of the
exponent $p$. The a priori estimates of Section
\ref{sec:a-priori} involve higher-order derivatives of the
approximate solutions. This is where we make use of the
interpolation inequalities of Section \ref{sec:prelim} to obtain
the global higher integrability of the gradient which,
in turn, yields uniform boundedness of the
$L^{q(\cdot)}(Q_T)$-norms of the gradients of the approximate
solutions.

Theorems \ref{PPresu2} and \ref{th:exist-source} are proven in Section \ref{sec:reg-existence}. We show first that for every $\epsilon>0$ the constructed sequence of Galerkin's approximations contains a subsequence which converges to a strong solution $u_\epsilon$ of the regularized problem \eqref{eq:reg-prob}. The proof relies on the compactness and monotonicity of the fluxes. Existence of a solution to problem \eqref{eq:main} is established in a similar way. We show that the solutions of the regularized problem \eqref{eq:reg-prob} converge (up to a subsequence) to a solution of the problem \eqref{eq:main}.

\medskip

\noindent \textbf{Notation:} Throughout the rest of the text, the symbol $C$ will be used to denote the constants which can be calculated or estimated through the data but whose exact values is unimportant. The value of $C$ may vary from line to line even inside the same formula. Whenever it does not cause a confusion, we omit the arguments of the variable exponents of nonlinearity and the coefficients. We will use the shorthand notation
\[
|v_{xx}|^2=\sum_{i,j=1}^{N}|v_{x_ix_j}|^2.
\]

\section{Auxiliary propositions}
Until the end of this section,  the notation $p(\cdot)$, $q(\cdot)$, $a(\cdot)$ is used for functions not related to the exponents and coefficient in \eqref{eq:main} and \eqref{eq:reg-prob}.

\label{sec:prelim}
\begin{Lem}[Lemma 1.32, \cite{ant-shm-book-2015}]
\label{le:embed-A}
Let $\partial\Omega\in {\rm Lip}$ and ${ p} \in C^{0}(\overline{Q}_T)$. Assume that {$u\in L^{\infty}(0,T;L^2(\Omega))\cap W^{1,p(\cdot)}_0(Q_T)$} and
\[
\operatorname{ess}\sup_{(0,T)}\|u(\cdot, t)\|_{2,\Omega}^{2}+\int_{Q_T}|\nabla u|^{p(z)}\,dz=M<\infty.
\]
Then
\[
\|u\|_{p(\cdot),Q_T}\leq C, \qquad C=C(M,p^\pm,N,\omega),
\]
where $\omega$ is the modulus of continuity of the exponent ${p(\cdot)}$.
\end{Lem}
The proof in \cite{ant-shm-book-2015} is given for the case $\Omega=B_R(x_0)$. To adapt it to the general case, it is sufficient to  consider the zero continuation of $u$ to a circular cylinder containing $Q_T$.
\vspace{0.1cm}\\
Let us accept the notation
\begin{equation}
\label{eq:gamma}
\begin{split}
& \beta_{\epsilon}(\mathbf{s})=\epsilon^2+|\mathbf{s}|^2,
\\
& \gamma_{\epsilon}(z,\mathbf{s})=(\epsilon^2+|\mathbf{s}|^2)^{\frac{p(z)-2}{2}}+ a(z)(\epsilon^2+|\mathbf{s}|^2)^{\frac{q(z)-2}{2}},
\quad \mathbf{s}\in \mathbb{R}^N,\quad z\in Q_T,\quad \epsilon\in (0,1).
\end{split}
\end{equation}
With certain abuse of notation, we will denote by $\gamma_{\epsilon}(x,\mathbf{s})$ the same function but with the exponents $p$, $q$ and the coefficient $a$ depending on the variable $x\in \Omega$.

\begin{Lem}[Lemma 4.1, \cite{A-S}]
\label{le:integr} Let $\partial \Omega \in C^1$, $u\in
C^2(\overline{\Omega})$ and $u=0$ on $\partial \Omega$. Assume that
\begin{equation}
\label{eq:cond-embed}
\begin{split}
& p:\Omega\mapsto [p^-,p^+],\quad p^\pm=const,
\\
& \dfrac{2N}{N+2}<p^-,\quad p(\cdot)\in
C^{0}(\overline{\Omega}),\qquad
\operatorname{ess}\sup_{\Omega}|\nabla p|=L,
\\
& \int_\Omega\beta_{\epsilon}^{\frac{p(x)-2}{2}}(\nabla
u)|u_{xx}|^2\,dx<\infty,\qquad \int_{\Omega}u^2\,dx= M_0,\qquad
\int_{\Omega}|\nabla u|^{p(x)}\,dx= M_1.
\end{split}
\end{equation}
Then for every
\begin{equation}
\label{eq:ast}
\frac{2}{N+2}=:r_\ast<r<r^\ast:= \dfrac{4p^-}{p^-(N+2)+2N}
\end{equation}
and every $\delta\in (0,1)$
\begin{equation}
\label{eq:principal}
\int_{\Omega}\beta^{\frac{p(x)+r-2}{2}}_\epsilon(\nabla u)|\nabla
u|^2\,dx\leq \delta\int_{\Omega}\beta^{\frac{p(x)-2}{2}}_\epsilon(\nabla u)|u_{xx}|^{2}\,dx+C\left(1+\int_{\Omega}|\nabla
u|^{p(x)}\,dx\right)
\end{equation}
with an independent of $u$ constant $C=C(\partial
\Omega,\delta,p^\pm,N,r,M_0,M_1)$.
\end{Lem}

\begin{thm}[Theorem 4.1, \cite{A-S}]
\label{th:integr-par} Let $\partial \Omega \in C^1$, $u\in
C^0([0,T];C^2(\overline{\Omega}))$ and $u=0$ on $\partial \Omega\times [0,T]$. Assume that
\begin{equation}
\label{eq:cond-embed-par}
\begin{split}
& p:Q_T\mapsto [p^-,p^+],\quad p^\pm=const,
\\
& \text{$p \in
C^{0}(\overline{Q}_T)$ with the modulus of continuity $\omega$},
\\
& \dfrac{2N}{N+2}<p^-,\qquad
\qquad
\operatorname{ess}\sup_{Q_T}|\nabla p|=L,
\\
& \int_{Q_T}\beta_{\epsilon}^{\frac{p(z)-2}{2}}(\nabla
u)|u_{xx}|^2\,dz<\infty,\qquad \sup_{(0,T)}\|u(t)\|_{2,\Omega}^2= M_0,\qquad
\int_{Q_T}|\nabla u|^{p(z)}\,dz= M_1.
\end{split}
\end{equation}
Then for every
\[
\frac{2}{N+2}=r_\ast<r<r^\ast= \dfrac{4p^-}{p^-(N+2)+2N}
\]
and every $\delta\in (0,1)$ the function $u$ satisfies the inequality

\begin{equation}
\label{eq:principal-par}
\int_{Q_T}\beta^{\frac{p(z)+r-2}{2}}_\epsilon(\nabla u)|\nabla
u|^2\,dz\leq \delta\int_{Q_T}\beta_{\epsilon}^{\frac{p(z)-2}{2}}(\nabla
u)|u_{xx}|^{2}\,dz+C\left(1+\int_{Q_T}|\nabla
u|^{p(z)}\,dz\right)
\end{equation}
with an independent of $u$ constant $C=C(N,\partial
\Omega,T,\delta,p^\pm,\omega,r,M_0,M_1)$.
\end{thm}

\begin{Lem}
\label{le:trace-old}
Let $\Omega\subset \mathbb{R}^N$, $N\geq 2$ be a bounded domain with the boundary $\partial\Omega\in C^{2}$, and $ {a \in W^{1,\infty}(\Omega)}$ be a given nonnegative function. Assume that $v\in W^{3,2}(\Omega)\cap W^{1,2}_0(\Omega)$ and denote

\begin{equation}
\label{eq:K}
K=\int_{\partial\Omega}a(x)(\epsilon^2+|\nabla v|^2)^{\frac{p(x)-2}{2}}\left(\Delta v \,(\nabla v\cdot \mathbf{n})-\nabla (\nabla v\cdot \mathbf{n})\cdot \nabla v\right)\,dS,
\end{equation}
where $\mathbf{n}$ stands for the exterior normal to ${\partial \Omega}$. There exists a constant $L=L(\partial\Omega)$ such that
\[
K\leq L\int_{\partial\Omega}a(x)(\epsilon^2+|\nabla v|^2)^{\frac{p(x)-2}{2}}|\nabla v|^2\,dS.
\]
\end{Lem}

Lemma \ref{le:trace-old} follows from the  well-known assertions, see, e.g., \cite[Ch.1, Sec.1.5]{Lad} for the case $a\equiv 1$, $N\geq 2$, or \cite[Lemma A.1]{AKS} for the case of an arbitrary dimension. Fix an arbitrary point $\xi\in \partial\Omega$ and introduce the local coordinate system $\{y\}$ with the origin $\xi$. The system is chosen so that $y_N$ coincides with the direction $\mathbf{n}$. There is a neighborhood of $\xi$ where $\partial\Omega$ is represented in the form $y_N=\omega(y_1,\ldots,y_{N-1})$ with a twice differentiable function $\omega$. In the local coordinates
\[
I_{\partial\Omega}\equiv \Delta v \,(\nabla v\cdot \mathbf{n})-\nabla (\nabla v\cdot \mathbf{n})\cdot \nabla v=\sum_{i=1}^{N-1}\left(D^2_{y_iy_i}w D_{y_N}w-D^2_{y_iy_N}w D_{y_i}w\right),
\]
where $w(y)=v(x)$, and

\[
I_{\partial\Omega}(\xi)=-\left(D_{y_N}w(0)\right)^2 \sum_{i=1}^{N-1}D_{y_iy_i}^2\omega(0)=-(\nabla v(\xi)\cdot \mathbf{n})^2\sum_{i=1}^{N-1}D_{y_iy_i}^2\omega(0).
\]
Since $\omega$ is two times differentiable, then $|I_{\partial\Omega}(\xi)|\leq C |\nabla v(\xi)|^2$ with a constant $C$ depending only on $N$ and $\sup |D^{2}_{y_iy_j}\omega(y)|$. Estimate \eqref{eq:K} follows because $\xi\in \partial \Omega$ is arbitrary.

\begin{Lem}
\label{le:trace-1} Let $\partial\Omega$ be a Lipschitz-continuous
surface and ${a(\cdot)}$ be a nonnegative function on $\overline{\Omega}$. Assume that $a,q\in W^{1,\infty}(\Omega)$,  with
\[
\|\nabla q\|_{\infty,\Omega}\leq L<\infty , \quad  \|\nabla a\|_{\infty,\Omega}\leq L_0<\infty.
\]
There exists
a constant $\delta=\delta(\partial\Omega)$ such that for every
$u\in W^{1,q(\cdot)}(\Omega)$

\begin{equation}
\label{eq:trace-1} \begin{split}
\delta\int_{\partial\Omega} a(x) &  (\epsilon^2+|u|^2)^{\frac{q(x)-2}{2}}|u|^2\,dS
\\
&
\leq
C\int_{\Omega}\left(a(x)|u|^{q(x)-1}|\nabla
u|+a(z)|u|^{q(x)}|\ln|u||+|u|^{q(x)}+1\right)\,dx
\end{split}
\end{equation}
with a constant $C=C(q^+,L,L_0,N,\Omega)$.
\end{Lem}

\begin{proof} By \cite[Lemma 1.5.1.9]{GRI} there exists $\delta>0$
and $\mu\in (C^{\infty}(\overline{\Omega}))^N$ such that
$\mu\cdot\mathbf{n}\geq \delta$ a.e. on $\partial\Omega$. By the
Green formula
\[
\begin{split}
\delta &\int_{\partial\Omega}a(x) |u|^{q(x)}\,dS  \leq
\int_{\partial\Omega} a(x)|u|^{q(x)}(\mu\cdot \mathbf{n})\,dS
 =\int_{\Omega} \operatorname{div}(a(x) |u|^{q(x)}\mu)\,dx
\\
& = \int_{\Omega}\left[a(x)\left(q(x)|u|^{q(x)-2}u (\nabla u \cdot
\mu)+|u|^{q(x)}\ln |u|(\nabla
q \cdot \mu) +|u|^{q(x)}\operatorname{div}\mu\right) + |u|^{q(x)} (\nabla a \cdot \mu)\right]\,dx
\\
& \leq q^+\max_{\Omega}|\mu|\int_{\Omega} a(x) |u|^{q(x)-1}|\nabla
u|\,dx+\|\nabla
q\|_{\infty,\Omega}\max_{\Omega}|\mu|\int_{\Omega} a(x)|u|^{q(x)}|\ln|u||\,dx
\\
& \qquad +
\max_{\Omega}|\operatorname{div}\mu|\int_{\Omega} a(x) |u|^{q(x)}\,dx + \max_{\Omega}|\mu| \||\nabla a|\|_{L^\infty(\Omega)}  \int_{\Omega} |u|^{q(x)}\,dx
\\
& \leq C\int_{\Omega}\left( a(x) |u|^{q(x)-1}|\nabla
u|+ a(x) |u|^{q(x)}|\ln|u||+|u|^{q(x)}\right)\,dx
\end{split}
\]
with $C=C(N,q^+,L,L_0,\Omega)$. This inequality implies \eqref{eq:trace-1} because

\[
a(x)(\epsilon^2+| u|^2)^{\frac{q(x)-2}{2}}| u|^{2}\leq a(x)(\epsilon^2+| u|^2)^{\frac{q(x)}{2}}\leq C+a(x)|u|^{q(x)}
\]
with an independent of $u$ constant $C$.
\end{proof}

\begin{cor}
\label{le:trace-2} Under the conditions of Lemma \ref{le:trace-1},
for every $\lambda\in (0,1)$ and $\epsilon\in (0,1)$
\begin{equation}
\label{eq:trace-2}
\begin{split}
\int_{\partial\Omega}  a(x)(\epsilon^2+|u|^2)^{\frac{q(x)-2}{2}}|u|^2\,dS &\leq \lambda
\int_{\Omega} a(x) (\epsilon^{2}+|u|^{2})^{\frac{q(x)-2}{2}}|\nabla
u|^{2}\,dx
\\
& +  L_0 \int_{\Omega} |u|^{q(x)}\,dx +
L\int_{\Omega}a(z) |u|^{q(x)}|\ln |u||\,dx +K
\end{split}
\end{equation}
with independent of $u$ constants $K$, $L$, $L_0$.
\end{cor}

\begin{proof}
We transform the first term on the right-hand side of \eqref{eq:trace-1} using the Cauchy inequality:
\[
\begin{split}
a|u|^{q-1}|\nabla u| & \leq (a(\epsilon^2+|u|^2)^{\frac{q-2}{2}}|\nabla u|^2)^{\frac{1}{2}}(a(\epsilon^2+|u|^2)^{\frac{q}{2}})^{\frac{1}{2}}
\\
&
\leq \lambda a(\epsilon^2+|u|^2)^{\frac{q-2}{2}}|\nabla u|^2+C a (\epsilon^2+|u|^2)^{\frac{q}{2}}.
\end{split}
\]
\end{proof}
\begin{thm}
\label{th:trace-main} Let $\partial\Omega\in C^2$, $u\in C^2(\overline{\Omega})$ and $u=0$ on $\partial \Omega$. Assume that ${a(\cdot)}$  satisfies the conditions of Lemma \ref{le:trace-1}, ${p(\cdot)}$ satisfies the conditions of  Lemma \ref{le:integr}, and

\begin{equation}
\notag
\begin{split}
q:\Omega\mapsto [q^-,q^+]\subset \left(
 \dfrac{2N}{N+2},\infty\right), \qquad q\in
W^{1,\infty}(\Omega),\qquad
\operatorname{ess}\sup_{\Omega}|\nabla q|= L.
\end{split}
\end{equation}
If {for a.e. $x\in \Omega$}
\[
\text{$q(x) < p(x) + r $ with $\frac{2}{N+2}< r<\dfrac{4p^-}{p^-(N+2)+2N}$},
\]
then for every $\lambda\in (0,1)$
\begin{equation}
\label{eq:trace-3} \int_{\partial \Omega} \gamma_{\epsilon}(x,\nabla u)|\nabla u|^{2}\,dS\leq
\lambda \int_{\Omega} \gamma_{\epsilon}(x,\nabla u) |u_{xx}|^2\,dx+C\left(1+\int_{\Omega} |\nabla
u|^{p(x)}\,dx\right)
\end{equation}
with a constant $C$ depending on $\lambda$ and the constants
$p^\pm$, $N$, $L$, $L_0$, but independent
of $u$.
\end{thm}

\begin{proof}
Applying \eqref{eq:trace-2} to $|\nabla u|$ we obtain
\begin{equation}
\label{eq:trace-4}
\begin{split}
\int_{\partial \Omega}  a(x)(\epsilon^2+ |\nabla u|^2)^{\frac{q(x)-2}{2}}|\nabla u|^2\,dS &\leq \lambda
\int_{\Omega} a(x)(\epsilon^2+|\nabla
u|^{2})^{\frac{q(x)-2}{2}}|u_{xx}|^2\,dx
\\
& + L_0 \int_{\Omega} |\nabla u|^{q(x)}\,dx+ L\int_{\Omega}|\nabla
u|^{q(x)}\ln|\nabla u||\,dx+K
\end{split}
\end{equation}
with independent of $u$ constants $L,K, L_0$. {Choose $0<r_1<r_2<r^\ast$ so small that $q(x)+r_1 < p(x) + r_2$ and
\begin{equation*}
\begin{split}
|\nabla u|^{q(x)}|\ln |\nabla u|| &\leq \left\{
         \begin{alignedat}{2}
             {} |\nabla u|^{q(x)+r_1} ( |\nabla u|^{-r_1} |\ln |\nabla u||) \leq C(r_1,q^+) |\nabla u|^{q(x)+r_1}
             & {}
             && \quad\mbox{if}\ |\nabla u| \geq 1,
             \\
             |\nabla u|^{q^-} |\ln |\nabla u|| \leq C(q^-) \quad \quad \quad \quad \quad \quad \quad \quad \quad \quad \quad \quad & {}
             && \quad\mbox{if}\ |\nabla u| \in (0,1).
          \end{alignedat}
     \right.\\
     & \leq  C\left(1+|\nabla u|^{q(x)+r_1}\right)
     \end{split}
\end{equation*}
with a constant $C$ independent of $u$}. Thus, there exists a constant $C$ such that

\[
\text{$|\nabla u|^{q(x)}|\ln|\nabla u||\leq C(1+|\nabla
u|^{q(x)+r_1})\leq C(1+|\nabla u|^{p(x)+r_2})$ in $\Omega$}.
\]
Using this inequality and then applying Lemma \ref{le:integr} we continue \eqref{eq:trace-4} as follows:
\[
\begin{split}
\int_{\partial \Omega} & a(x)(\epsilon^2+ |\nabla u|^2)^{\frac{q(x)-2}{2}}|\nabla u|^2\,dS
\\
& \leq \lambda
\int_{\Omega} a(x)(\epsilon^2+|\nabla
u|^{2})^{\frac{q(x)-2}{2}}|u_{xx}|^2\,dx
\quad +C\left(1+\int_{\Omega} |\nabla u|^{p(x)+r_2}\,dx\right)
\\
& \leq \lambda
\int_{\Omega} a(x)(\epsilon^2+|\nabla
u|^{2})^{\frac{q(x)-2}{2}}|u_{xx}|^2\,dx +\lambda \int_{\Omega} (\epsilon^2+|\nabla
u|^{2})^{\frac{p(x)-2}{2}}|u_{xx}|^2\,dx
\\
&
\qquad +C\left(1+\int_{\Omega} |\nabla u|^{p(x)}\,dx\right)
\\
& = \lambda \int_{\Omega}\gamma_{\epsilon}(x,\nabla u)|u_{xx}|^2\,dx+C \left(1+\int_{\Omega} |\nabla u|^{p(x)}\,dx\right).
\end{split}
\]
Adding to this inequality the inequality corresponding to $q=p$ and $a\equiv 1$, we arrive at \eqref{eq:trace-3}.
\end{proof}

Since $\mathcal{P}_m\subset C^1(\overline{\Omega})\cap H^3_0(\Omega)$, the interpolation inequalities of this section remain true for every function $w\in \mathcal{P}_m$, $m\in \mathbb{N}$.

\section{Regularized problem}
\label{sec:reg-problem}

Given $\epsilon>0$, let us consider the following family of regularized double phase parabolic equations:
\begin{equation}\label{eq:reg-prob}
\left\{
         \begin{alignedat}{2}
             {} \partial_t u - \div(\gamma_{\epsilon}(z,\nabla u)  \nabla u)
             & {}= F(z,u)
             && \quad\mbox{ in } \, Q_T,
             \\
             u & {}= 0
             && \quad\mbox{ on }\, \Gamma_T,
             \\
             u(0,.)&{}= u_0
             && \quad\  \mbox{in}\, \ \Omega, \ \epsilon \in (0,1),
          \end{alignedat}
     \right.
\end{equation}
where $F(z,u)$ is defined in \eqref{eq:source}, $\gamma_\epsilon(z,\mathbf{s})$ is introduced in \eqref{eq:gamma}, and $\gamma_{\epsilon}(z, \nabla u)\nabla u$ is the regularized flux function.

\subsection{Galerkin's method}
Let $\epsilon>0$ be a fixed parameter. The sequence $\{u^{(m)}_\epsilon\}$ of finite-dimensional Galerkin's approximations for the solutions of the regularized problem \eqref{eq:reg-prob} is sought in the form
 \begin{equation}\label{eq:coeff}
 u^{(m)}_\epsilon(x,t)= \sum_{j=1}^m u_{j}^{(m)}(t) \phi_j(x)
 \end{equation}
 where $\phi_j \in W^{1,2}_0(\Omega)$ and $ \lambda_j>0$ are the eigenfunctions and the corresponding eigenvalues of {problem \eqref{eq:eigen}}.
 %\begin{equation}
% \label{eq:eigen}
% (\nabla \phi_j, \nabla \psi)_{2,\Omega}= \lambda_{j} (\phi_j, \psi)_{2, \Omega}\quad \forall \psi \in W^{1,2}_0(\Omega).
% \end{equation}
%The systems $\{\phi_j\}$ and $\{\lambda_j^{-\frac{1}{2}} \phi_j\}$ are the orthogonal bases of $L^2(\Omega)$ and $W_0^{1,2}(\Omega).$
The coefficients $u_j^{(m)}(t)$ are characterized as the solutions of the Cauchy problem for the system of $m$ ordinary differential equations
\begin{equation}\label{system}
   \left\{
         \begin{alignedat}{2}
             {} (u^{(m)}_j)'(t)
             & {}= - \int_{\Omega} \gamma_{\epsilon}(z, \nabla u_\epsilon^{(m)}) \nabla u^{(m)}_\epsilon \cdot \nabla \phi_j ~dx + \int_{\Omega} F(z,u_{\epsilon}^{(m)}) \phi_j ~dx
             && ,
             \\
             u^{(m)}_j(0)
             & {}= (u_0^{(m)}, \phi_j)_{2, \Omega}, \quad j=1,2,\dots,m,
             &&
          \end{alignedat}
     \right.
\end{equation}
where $\gamma_{\epsilon}$ is defined in \eqref{eq:gamma}. {By Lemma \ref{le:density-3} the functions $u_0^{(m)}\in \mathcal{P}_m$ can be chosen so that}

\begin{equation}
\label{eq:choice}
\begin{split}
& u_0^{(m)}= \sum_{j=1}^m (u_0, \phi_j)_{2, \Omega} \phi_j \in \operatorname{span}\{\phi_1, \phi_2, \dots, \phi_m\},
\\ & \text{$u_0^{(m)} \to u_0$} \quad \begin{array}{ll} \text{in $W_0^{1,2}(\Omega)$} & \text{if $\max_{\overline{\Omega}}q(x,0)\leq 2$},
\\
\text{{in $\mathcal{V}_{0}(\Omega)$}} &
\text{if $\max_{\overline{\Omega}}q(x,0)>2$.}
\end{array}
\end{split}
\end{equation}
By the Carath\'eodory existence theorem, for every finite $m$
system \eqref{system} has a solution $(u_1^{(m)}, u_2^{(m)},
\dots, u_m^{(m)} )$ in the extended sense on an interval
$(0,T_m)$, the functions $u_i^{(m)}(t)$ are absolutely continuous
and differentiable a.e. in $(0,T_m)$. {The a priori
estimates \eqref{eq:ineq-0}, \eqref{timederiest} in the case $b\equiv 0$, and \eqref{eq:ineq-0-source}, \eqref{eq:timederiest-source} in the case $b\not\equiv 0$,}  show that for every $m$ the function
$u_{\epsilon}^{(m)}(x,T_m)$ {belongs to $\operatorname{span}\{\phi_1,\ldots,\phi_m\}$} and satisfies the estimate
\[
\begin{split}
\|\nabla u_\epsilon^{(m)}(\cdot,T_m)\|_{2,\Omega}^2  & +
\mathcal{F}(u_\epsilon(\cdot,T_m),T_m) \leq {C+\|f_0\|_{2,Q_{T}}^2} + \|\nabla u_0^{(m)}\|_{2,\Omega}^2
+\mathcal{F}(u_0^{(m)},0)
\end{split}
\]
with the function $\mathcal{F}$ defined in \eqref{eq:F} and a constant $C$ independent of $m$ and $\epsilon$. This estimate allows one to continue each of $u_\epsilon^{(m)}$ to the maximal existence interval $(0,T)$.

\section{A priori estimates}
\label{sec:a-priori}
\subsection{A priori estimates I: the case $b\equiv 0$}

\begin{Lem}\label{1st}
Let $\Omega$ be a bounded domain {with the boundary $\partial\Omega\in {\rm Lip}$}, {$p(\cdot),q(\cdot)$} satisfy \eqref{assum1}, {$a(\cdot)$} satisfies \eqref{eq:a},
$u_0 \in L^2(\Omega)$ and $f_0 \in L^2(Q_T)$. If $ b \equiv 0$, then $u^{(m)}_\epsilon$ satisfies the estimates
\begin{equation}\label{secderiboun}
\begin{split}
\sup_{t \in (0,T)} \|u^{(m)}_\epsilon(\cdot,t)\|^2_{2,\Omega} + \int_{Q_T} \gamma_{\epsilon}(z, \nabla u_\epsilon^{(m)}) |\nabla u^{(m)}_\epsilon|^2 ~dz \leq C_1 {\rm e}^{ T} (\|f_0\|^2_{2,Q_T}  + \|u_0\|^2_{2,\Omega})
\end{split}
\end{equation}
and
\begin{equation}\label{gradbound}
\begin{split}
\int_{Q_{T}} \left(|\nabla u^{(m)}_\epsilon|^{p(z)} + a(z) |\nabla u^{(m)}_\epsilon|^{q(z)}\right) ~dz \leq C_2
\int_{Q_T}  \gamma_{\epsilon}(z, \nabla u_\epsilon^{(m)}) |\nabla u^{(m)}_\epsilon|^2 ~dz + C_3
\end{split}
\end{equation}
where the constants $C_i$ are independent of $\epsilon$ and $m$.
\end{Lem}
\begin{proof}
By multiplying $j^{\text{th}}$ equation of \eqref{system} by $u_j^{(m)}(t)$ and then by summing up the results for $j=1,2, \dots, m$, we obtain
\begin{equation}\label{est01}
\begin{split}
\frac{1}{2} \frac{d}{dt} \|u^{(m)}_\epsilon{(\cdot, t)}\|^2_{2,\Omega} &=\sum_{j=1}^m u_j^{(m)}(t) (u^{(m)}_j)'(t) = - \sum_{j=1}^m  u_j^{(m)}(t) \int_{\Omega} \gamma_{\epsilon}(z, \nabla u_\epsilon^{(m)}) \nabla u^{(m)}_\epsilon . \nabla \phi_j \,~dx\\
&\quad \quad +  \sum_{j=1}^m \int_{\Omega} f_0(x,t) \phi_j(x) u_j^{(m)}(t) ~dx \\
& = - \int_{\Omega} \gamma_{\epsilon}(z, \nabla u_\epsilon^{(m)}) |\nabla u^{(m)}_\epsilon|^2 ~dx + \int_{\Omega} f_0(x,t) u^{(m)}_\epsilon ~dx.
\end{split}
\end{equation}
Using the Cauchy inequality, we obtain
\begin{equation}
\label{eq:option-1}
\begin{split}
\frac{1}{2} \frac{d}{dt} \|u^{(m)}_\epsilon {(\cdot, t)}\|^2_{2,\Omega} + \int_{\Omega} \gamma_{\epsilon}(z, \nabla u_\epsilon^{(m)}) |\nabla u^{(m)}_\epsilon|^2 ~dx &\leq  \frac{1}{2} \|f_0{(\cdot, t)}\|^2_{2,\Omega} + \frac{1}{2} \|u^{(m)}_\epsilon{(\cdot, t)}\|^2_{2,\Omega}.
\end{split}
\end{equation}
Now, rewriting the last inequality in the equivalent form
\begin{equation*}
\begin{split}
\frac{1}{2} \frac{d}{dt} \left( {\rm e}^{- t}\|u^{(m)}_\epsilon{(\cdot, t)} \|^2_{L^2(\Omega)}\right) &+ {\rm e}^{-t} \int_{\Omega}\gamma_{\epsilon}(z, \nabla u_\epsilon^{(m)}) |\nabla u^{(m)}_\epsilon|^2 ~dx \leq  \frac{{\rm e}^{-t}}{2} \|f_0{(\cdot, t)}\|^2_{2,\Omega}
\end{split}
\end{equation*}
and integrating with respect to $t$, we arrive at the inequality
\begin{equation}\label{est03}
\begin{split}
\sup_{t \in (0,T)} \|u^{(m)}_\epsilon(\cdot,t)\|^2_{L^2(\Omega)} &+ \int_{Q_T} \gamma_{\epsilon}(z, \nabla u_\epsilon^{(m)}) |\nabla u^{(m)}_\epsilon|^2 ~dx ~dt \leq C {\rm e}^{ T} \left(\|f_0\|^2_{2,Q_T}  + \|u_0\|^2_{2,\Omega}\right)
\end{split}
\end{equation}
where the constant $C$ is independent of $\epsilon$ and $m$. Since $a(z)$ is a nonnegative bounded function, the second assertion follows from \eqref{est03} and the inequality
\begin{equation}\label{bdd1}
\begin{split}
a(z)|\nabla u_{\epsilon}^{(m)}|^{q(z)} & \leq a(z) \left(\epsilon^2+|\nabla
u_{\epsilon}^{(m)}|^2\right)^{\frac{q(z)}{2}}
\\
&
\leq
\begin{cases} 2 a(z) \left(\epsilon^2+|\nabla
u_{\epsilon}^{(m)}|^2\right)^{\frac{q(z)-2}{2}}|\nabla
u_\epsilon^{(m)}|^2 & \text{if $|\nabla u_{\epsilon}^{(m)}|\geq
\epsilon$},
\\
(2\epsilon^2)^{\frac{q(z)}{2}} a(z) \leq 2^{\frac{q^+}{2}} a(z)  &
\text{otherwise}.
\end{cases}
\end{split}
\end{equation}
\end{proof}
\begin{Lem}\label{second}
Let $\Omega$ be a bounded domain with {$\partial \Omega\in C^k$, $k\geq 2+[\frac{N}{2}]$}. Assume that {$p(\cdot),q(\cdot)$ satisfies  \eqref{assum1}, \eqref{eq:Lip-p-q}, {\eqref{eq:oscillation}} and and $a(\cdot)$} satisfy \eqref{eq:a}. If $u_0 \in W_0^{1,2}(\Omega)$, $f_0\in L^2((0,T);W_0^{1,2}(\Omega))$ and $b \equiv 0$, then for a.e. $t\in (0,T)$ the following inequality holds:
\begin{equation}
\label{eq:aux-est}
\begin{split}
\frac{1}{2} \frac{d}{dt} \|\nabla
u^{(m)}_\epsilon{(\cdot, t)}\|^2_{2,\Omega} & + C_0 \int_{\Omega}
\gamma_{\epsilon}(z, \nabla u_\epsilon^{(m)})
|(u^{(m)}_\epsilon)_{xx}|^2 \,dx
\\
& \leq  C_1\left(1 + \int_{\Omega} |\nabla u_\epsilon^{(m)}|^{p(z)}\,dx  + \|\nabla u_{\epsilon}^{(m)}{(\cdot, t)}\|_{2,\Omega}^{2}+
\|f_0{(\cdot, t)}\|_{W^{1,2}_0(\Omega)}^2\right)
\end{split}
\end{equation}
with independent of $m$ and $\epsilon$ constants $0<C_0<\min\{p^--1,1\}$ and $C_1>0$.
\end{Lem}

\begin{proof}
Let us multiply each of equations in \eqref{system} by $\lambda_j
u_j^{(m)}$ and sum up the results for $j=1,2,\dots,m$:
\begin{equation}
\label{eq:prime}
\begin{split}
\frac{1}{2} & \frac{d}{dt} \|\nabla
u^{(m)}_\epsilon {(\cdot, t)}\|^2_{2,\Omega} =\sum_{j=1}^m \lambda_j
(u_j^{(m)})'(t) u_j^{(m)}(t)
\\
& = \sum_{j=1}^m \lambda_j u_j^{(m)} \int_{\Omega}
\div(\gamma_{\epsilon}(z, \nabla u_\epsilon^{(m)}) \nabla u_\epsilon^{(m)})\ \phi_j \,dx + \sum_{j=1}^m \lambda_j
u_{j}^{(m)} \int_{\Omega} f_0(x,t) \phi_j \,dx\\ &= - \int_{\Omega}
\div(\gamma_{\epsilon}(z, \nabla u_\epsilon^{(m)}) \nabla u_\epsilon^{(m)}) \ \Delta u^{(m)}_\epsilon \,dx - \int_{\Omega} f_0(x,t)
\Delta u^{(m)}_\epsilon \,dx.
\end{split}
\end{equation}
{Since $\partial\Omega\in C^{k}$ with $k\geq 2+[\frac{N}{2}]$, then $u_\epsilon^{(m)}(\cdot,t)\in \mathcal{P}_m\subset H_0^{3}(\Omega)\cap C^{1}(\overline{\Omega})$} . Therefore the first term on the right-hand of \eqref{eq:prime} can be transformed by means of the Green formula:
\[
\begin{split}
- \int_{\Omega}  &  \mathrm{div}\left( \gamma_{\epsilon}(z, \nabla u_\epsilon^{(m)}) \nabla u_\epsilon^{(m)}\right) \,\Delta u_\epsilon^{(m)}\,dx
\\
&  = - \int_{\Omega}\left( \sum_{k=1}
^{N}(u^{(m)}_\epsilon)_{x_kx_k}\right) \left( \sum_{i=1}^{N}\left(
\gamma_{\epsilon}(z, \nabla u_\epsilon^{(m)}) (
u^{(m)}_\epsilon)_{x_i}\right)_{x_i} \right)  \,dx \\
&  = - \int_{\partial\Omega}\Delta u^{(m)}_\epsilon \gamma_{\epsilon}(z, \nabla u_\epsilon^{(m)})  (\nabla
u^{(m)}_\epsilon\cdot\mathbf{n})\,dS
+ \int_{\Omega}\sum_{k,i=1}
^{N}(u^{(m)}_\epsilon)_{x_kx_kx_i}\gamma_{\epsilon}(z, \nabla u_\epsilon^{(m)}) (u^{(m)}_\epsilon)_{x_i}\,dx\\
&  = - \int_{\partial\Omega} \gamma_{\epsilon}(z, \nabla u_\epsilon^{(m)}) \sum_{k,i=1}
^{N}\left((u^{(m)}_\epsilon)_{x_kx_k}
(u^{(m)}_\epsilon)_{x_i}n_{i}-(u^{(m)}_\epsilon)_{x_kx_i}
(u^{(m)}_\epsilon)_{x_i} n_{k}\right)  \, dS\\
&  \qquad -
\int_{\Omega}\sum_{k,i=1}^{N}(u^{(m)}_\epsilon)_{x_kx_i}\left(
\gamma_{\epsilon}(z, \nabla u_\epsilon^{(m)})  (u^{(m)}_\epsilon)_{x_i}\right)_{x_k}  \,dx\\
& = - \int_{\Omega} \gamma_{\epsilon}(z, \nabla u_\epsilon^{(m)}) |(u_{\epsilon}^{(m)})_{xx}|^2\,dx +J_{1}+J_2+J_{\partial
\Omega} + J_a,
\end{split}
\]
where $\mathbf{n}=(n_1,\ldots,n_N)$ is the outer normal vector to $\partial\Omega$,

\begin{equation*}
\begin{split}
J_1:&= \int_{\Omega} (2-p(z)) (\epsilon^2 + |\nabla
u^{(m)}_\epsilon|^2)^{\frac{p(z)-2}{2}-1} \left(\sum_{k=1}^N
\left(\nabla u^{(m)}_\epsilon \cdot \nabla
(u^{(m)}_\epsilon)_{x_k} \right)^2 \right)\,dx\\
& \qquad + \int_{\Omega}  (2-q(z)) a(z) (\epsilon^2 + |\nabla
u^{(m)}_\epsilon|^2)^{\frac{q(z)-2}{2}-1} \left(\sum_{k=1}^N
\left(\nabla u^{(m)}_\epsilon \cdot \nabla
(u^{(m)}_\epsilon)_{x_k} \right)^2 \right)\,dx,
\end{split}
\end{equation*}
\begin{equation*}
\begin{split}
J_2&= - \int_{\Omega} \sum_{k,i=1}^N (u^{(m)}_\epsilon)_{x_k x_i }
(u^{(m)}_\epsilon)_{x_i}   (\epsilon^2 + |\nabla
u^{(m)}_\epsilon|^2)^{\frac{p(z)-2}{2}} \frac{p_{x_k}}{2}
\ln(\epsilon^2 + |\nabla u^{(m)}_\epsilon|^2)\,dx\\
& \quad \quad - \int_{\Omega} \sum_{k,i=1}^N (u^{(m)}_\epsilon)_{x_k x_i }
(u^{(m)}_\epsilon)_{x_i}  a(z) (\epsilon^2 + |\nabla
u^{(m)}_\epsilon|^2)^{\frac{q(z)-2}{2}} \frac{q_{x_k}}{2}
\ln(\epsilon^2 + |\nabla u^{(m)}_\epsilon|^2)\,dx,
\end{split}
\end{equation*}
\begin{equation*}
\begin{split}
J_{\partial \Omega}= - \int_{\partial\Omega} \gamma_{\epsilon}(z, \nabla u_\epsilon^{(m)})  \left(\Delta
u^{(m)}_\epsilon (\nabla u^{(m)}_\epsilon \cdot {\bf n})- \nabla
u^{(m)}_\epsilon \cdot \nabla (\nabla u^{(m)}_\epsilon \cdot{\bf n
})\right)\,dS,
\end{split}
\end{equation*}
\begin{equation*}
\begin{split}
J_{a}= - \int_{\Omega} \sum_{i,k=1}^N a_{x_k} (u^{(m)}_\epsilon)_{x_i} (\epsilon^2 + |\nabla u^{(m)}_\epsilon|^2)^{\frac{q(z)-2}{2}} {(u_\epsilon^{(m)})_{x_k x_i}}.
\end{split}
\end{equation*}
Substitution into \eqref{eq:prime} leads to the inequality
\begin{equation}\label{mainest}
\begin{split}
\frac{1}{2} \frac{d}{dt} \|\nabla
u^{(m)}_\epsilon{(\cdot, t)}\|^2_{2,\Omega} & + \int_{\Omega}  \gamma_{\epsilon}(z, \nabla u_\epsilon^{(m)})
|(u^{(m)}_\epsilon)_{xx}|^2\,dx
\\
& = J_1 + J_2 + J_{\partial\Omega} + J_a - \int_{\Omega}\nabla f_0\cdot
\nabla u_{\epsilon}^{(m)}\,dx
\\
& \leq J_1 + J_2 + J_{\partial\Omega} + J_a +\frac{1}{2}\|\nabla
u_{\epsilon}^{(m)}{(\cdot, t)}\|_{2,\Omega}^{2}+\frac{1}{2}\|f_0{ (\cdot, t)}\|_{W^{1,2}_{0}(\Omega)}^{2}.
\end{split}
\end{equation}
The terms on the right-hand side of \eqref{mainest} are estimated in three steps.

\medskip

\textbf{Step 1:} estimate on $J_1$. Since $a(z) \geq 0$ and {$p(z)<q(z)$ in $Q_T$}, the term $J_1$ is merged in the left-hand side. Indeed:
\[
\begin{split}
J_1 & = \int_{\{ x\in \Omega :\ p(z) \geq 2\}}(2-p(z))\ldots + \int_{\{ x\in \Omega :\  p(z) < 2\}}(2-p(z))\ldots \\
&\qquad + \int_{\{ x\in \Omega :\ q(z) \geq 2\}}(2-q(z))\ldots + \int_{\{ x\in \Omega :\  q(z) < 2\}}(2-q(z))\ldots
\\
&
\leq \int_{\{ x\in \Omega :\  p(z) < 2\}}
(2-p(z)) (\epsilon^2 + |\nabla
u^{(m)}_\epsilon|^2)^{\frac{p(z)-2}{2}-1} \left(\sum_{k=1}^N
\left(\nabla u^{(m)}_\epsilon \cdot \nabla
(u^{(m)}_\epsilon)_{x_k} \right)^2 \right)\,dx\\
& \qquad + \int_{\{ x\in \Omega :\  q(z) < 2\}}
(2-q(z)) a(z) (\epsilon^2 + |\nabla
u^{(m)}_\epsilon|^2)^{\frac{q(z)-2}{2}-1} \left(\sum_{k=1}^N
\left(\nabla u^{(m)}_\epsilon \cdot \nabla
(u^{(m)}_\epsilon)_{x_k} \right)^2 \right)\,dx,
\end{split}
\]
whence
\[
\begin{split}
|J_1| & \leq \max\{0,2-p^-\}\int_{\Omega} (\epsilon^2+|\nabla u_\epsilon^{(m)}|^2)^{\frac{p(z)-2}{2}}
|(u^{(m)}_\epsilon)_{xx}|^2\,dx
\\
&
\qquad + \max\{0,2-q^-\}\int_{\Omega} a(z)(\epsilon^2+|\nabla u_\epsilon^{(m)}|^2)^{\frac{q(z)-2}{2}}
|(u^{(m)}_\epsilon)_{xx}|^2\,dx.
\end{split}
\]

\medskip

\textbf{Step 2:} estimate on $J_2$. By the Cauchy inequality, for every $\delta_0>0$
\begin{equation}\label{est:J2}
\begin{split}
|J_2| &\leq \frac{1}{2} \|\nabla p \|_{\infty,\Omega} \int_{\Omega} \left( (\epsilon^2 + |\nabla u^{(m)}_\epsilon|^2)^{\frac{p(z)-2}{4}} \sum_{k,i=1}^N |(u^{(m)}_\epsilon)_{x_k x_i }| \right) \\
& \qquad \qquad\qquad  \times \left(|(u^{(m)}_\epsilon)_{x_i}|   |\ln(\epsilon^2 + |\nabla u^{(m)}_\epsilon|^2)|  (\epsilon^2 + |\nabla u^{(m)}_\epsilon|^2)^{\frac{p(z)-2}{4}} \right) ~dx\\
&\quad + \frac{1}{2} \|\nabla q \|_{\infty,\Omega} \int_{\Omega} \left( (a(z))^{\frac{1}{2}} (\epsilon^2 + |\nabla u^{(m)}_\epsilon|^2)^{\frac{q(z)-2}{4}} \sum_{k,i=1}^N |(u^{(m)}_\epsilon)_{x_k x_i }| \right) \\
& \qquad \qquad\qquad \times \left( (a(z))^{\frac{1}{2}} |(u^{(m)}_\epsilon)_{x_i}|   |\ln(\epsilon^2 + |\nabla u^{(m)}_\epsilon|^2)|  (\epsilon^2 + |\nabla u^{(m)}_\epsilon|^2)^{\frac{q(z)-2}{4}} \right) ~dx\\
& \leq \delta_0 \int_{\Omega} \gamma_{\epsilon}(z, \nabla u_\epsilon^{(m)})  \sum_{k,i=1}^N |(u^{(m)}_\epsilon)_{x_k x_i }|^2 ~dx \\
& \qquad + C_1 \int_{\Omega} \ln^2(\epsilon^2 + |\nabla u^{(m)}_\epsilon|^2)  \gamma_{\epsilon}(z,\nabla
u^{(m)}_\epsilon) |\nabla u^{(m)}_\epsilon|^2 \,dx
\end{split}
\end{equation}
with a constant $C_1=C_1(C^\ast,C^{\ast\ast}, N,\delta_0)$. Let us denote

\[
\mathcal{M}=C_1 \int_{\Omega} \ln^2(\epsilon^2 + |\nabla u^{(m)}_\epsilon|^2)  \gamma_{\epsilon}(z,\nabla
u^{(m)}_\epsilon) |\nabla u^{(m)}_\epsilon|^2 \,dx.
\]
For $\mu_1 \in (0,1)$ and $y>0$ the following inequality holds:

\begin{equation}\label{ineqq}
y^{\frac{p}{2}} \ln^2 y \leq \left\{
         \begin{alignedat}{2}
             {} y^{\frac{p + \mu_1}{2}} (y^{\frac{-\mu_1}{2}} \ln^2(y)) \leq C(\mu_1,p^+) (y^{\frac{p + \mu_1}{2}})
             & {}
             && \quad\mbox{if}\ y \geq 1,
             \\
             y^{\frac{p^-}{2}} \ln^2(y) \leq C(p^-) \quad \quad \quad \quad \quad \quad \quad \quad \quad & {}
             && \quad\mbox{if}\ y \in (0,1).
          \end{alignedat}
     \right.
\end{equation}
Let
\begin{equation}
\label{eq:r}
\text{$r_\ast= \dfrac{2}{N+2}$ and $r^\ast=\dfrac{4p^-}{p^-(N+2)+ 2N}$}.
\end{equation}
Take the numbers $r_1$, $r_2$ such that
\[
r_1 \in (r_\ast,r^\ast), \quad r_2 \in (0,1),\quad q(z) + r_2 \leq p(z) +r_1 < p(z)+ r^\ast
\]
and estimate $\mathcal{M}$ applying \eqref{ineqq}:
\[
\begin{split}
\mathcal{M} \leq C \left(1+ \int_{\Omega} (\epsilon^2 + |\nabla u^{(m)}_\epsilon|^2)^{\frac{p(z)+r_1-2}{2}}
|\nabla u^{(m)}_\epsilon|^2 \,dx + \int_{\Omega} a(z) (\epsilon^2 + |\nabla u^{(m)}_\epsilon|^2)^{\frac{q(z)+r_2-2}{2}}
|\nabla u^{(m)}_\epsilon|^2 \,dx \right)
\end{split}
\]
with a constant $C=C(C_1,r_1,r_2)$. Let us transform the integrand of the second integral using the following inequality:
\begin{equation}
\label{eq:notice}
\begin{split}
(\epsilon^2 + |\nabla u^{(m)}_\epsilon|^2)^{\frac{q(z)+r_2-2}{2}}
|\nabla u^{(m)}_\epsilon|^2 & \leq (\epsilon^2 + |\nabla u^{(m)}_\epsilon|^2)^{\frac{q(z)+r_2}{2}}
\leq 1+(\epsilon^2 + |\nabla u^{(m)}_\epsilon|^2)^{\frac{p(z)+r_1}{2}}
\\
& \leq 1+\begin{cases}
(2\epsilon^2)^{\frac{p(z)+r_1}{2}} & \text{if $|\nabla u^{(m)}_\epsilon|<\epsilon$},
\\
2(\epsilon^2 + |\nabla u^{(m)}_\epsilon|^2)^{\frac{p(z)+r_1-2}{2}}|\nabla u_\epsilon^{(m)}|^2 & \text{if $|\nabla u^{(m)}_\epsilon|\geq \epsilon$}.
\end{cases}
\end{split}
\end{equation}
Using \eqref{eq:notice} and the interpolation inequality of Lemma \ref{le:integr} we finally obtain
\begin{equation}\label{eq:logest1}
\begin{split}
\mathcal{M} & \leq C\left( {1+ } \int_{\Omega} (\epsilon^2 + |\nabla u^{(m)}_\epsilon|^2)^{\frac{p(z)+r_1-2}{2}}
|\nabla u^{(m)}_\epsilon|^2 \,dx \right)
\\
&
\leq  \delta_1 \int_{\Omega} (\epsilon^2 + |\nabla u_\epsilon^{(m)}|^2)^{\frac{p(z)-2}{2}} |(u_\epsilon^{(m)})_{xx}|^2 \,dx + C\left(1+ \int_{\Omega} |\nabla u_\epsilon^{(m)}|^{p(z)} \,dx \right)
\end{split}
\end{equation}
with any $\delta_1 \in (0,1)$ and $C=C(\delta_1)$. Gathering \eqref{est:J2} and \eqref{eq:logest1}, we finally obtain:
\[
|J_2| \leq (\delta_0 + \delta_1) \int_{\Omega} (\epsilon^2 + |\nabla u_\epsilon^{(m)}|^2)^{\frac{p(z)-2}{2}} |(u_\epsilon^{(m)})_{xx}|^2 \,dx +
C \left(1+ \int_{\Omega} |\nabla u_\epsilon^{(m)}|^{p(z)} \,dx \right)
\]
with a constant $C$ depending on $\delta_i$ and $\|a(\cdot,t)\|_{\infty,\Omega}$, but independent of $\epsilon$ and $m$.

\medskip

\textbf{Step 3:} estimates on $J_a$ and $J_{\partial \Omega}$. Let $\rho \in (r_\ast,r^\ast)$ be such that {$2q(z) -p(z)< p(z)+\rho < p(z) + r^\ast$}. Applying Young's inequality and \eqref{eq:notice} we obtain the estimate
{\[
\begin{split}
|J_{a}| &\leq \int_{\Omega} \sum_{i,k=1}^N |a_{x_k}| |(u^{(m)}_\epsilon)_{x_i}| (\epsilon^2 + |\nabla u^{(m)}_\epsilon|^2)^{\frac{q(z)-2}{2}} |(u^{(m)}_\epsilon)_{x_k x_i}| \,dx\\
& \leq \|\nabla a\|_{\infty,\Omega} \int_{\Omega} (\epsilon^2 + |\nabla u^{(m)}_\epsilon|^2)^{\frac{2q(z)-p(z)}{4}} \left( (\epsilon^2 + |\nabla u^{(m)}_\epsilon|^2)^{\frac{p(z)-2}{4}} |(u^{(m)}_\epsilon)_{xx}|\right) \,dx \\
& \leq \tilde{\delta} \int_{\Omega} (\epsilon^2 + |\nabla u^{(m)}_\epsilon|^2)^{\frac{p(z)-2}{2}} |(u^{(m)}_\epsilon)_{xx}|^2 \,dx + C(\tilde{\delta}) \int_{\Omega} (\epsilon^2 + |\nabla u^{(m)}_\epsilon|^2)^{\frac{2q(z)-p(z)}{2}} \,dx \\
& \leq \tilde{\delta} \int_{\Omega} (\epsilon^2 + |\nabla u^{(m)}_\epsilon|^2)^{\frac{p(z)-2}{2}} |(u^{(m)}_\epsilon)_{xx}|^2 \,dx + C' \left(1+\int_{\Omega}(\epsilon^2+|\nabla u_\epsilon^{(m)}|^2)^{\frac{p(z)+\rho}{2}}\,dx\right)\\
& \leq \tilde{\delta} \int_{\Omega} (\epsilon^2 + |\nabla u^{(m)}_\epsilon|^2)^{\frac{p(z)-2}{2}} |(u^{(m)}_\epsilon)_{xx}|^2 \,dx + C{''} \left(1+\int_{\Omega}(\epsilon^2+|\nabla u_\epsilon^{(m)}|^2)^{\frac{p(z)+\rho-2}{2}}|\nabla u_\epsilon^{(m)}|^2\,dx\right)
\end{split}
\]
}
where $C{''}= C{''}(\|\nabla a\|_{\infty,\Omega},N,q)$ is independent of $\epsilon$ and $m$. By Lemma \ref{le:integr} we obtain
\[
|J_a| \leq \delta_2 \int_{\Omega} (\epsilon^2 + |\nabla u_\epsilon^{(m)}|^2)^{\frac{p(z)-2}{2}} |(u_\epsilon^{(m)})_{xx}|^2 \,dx + C \left(1+ \int_{\Omega} |\nabla u_\epsilon^{(m)}|^{p(z)} \,dx \right)
\]
for any $\delta_2 \in (0,1)$ and a constant $C$ independent of $\epsilon$ and $m$.

To estimate $J_{\partial \Omega}$ we use Lemma \ref{le:trace-old} and Theorem \ref{th:trace-main}:
\begin{equation*}
\begin{split}
|J_{\partial \Omega}| & \leq  \left|\int_{\partial\Omega} \gamma_{\epsilon}(z,\nabla u_{\epsilon}^{(m)})\left(\Delta
u^{(m)}_\epsilon (\nabla u^{(m)}_\epsilon \cdot {\bf n})- \nabla
u^{(m)}_\epsilon \cdot \nabla (\nabla u^{(m)}_\epsilon \cdot{\bf n
})\right)\,dS \right|\\
& \leq C \int_{\partial \Omega} \gamma_{\epsilon}(z,\nabla u_{\epsilon}^{(m)})|\nabla u_\epsilon^{(m)}|^2 \,dS\\
& \leq \delta_3 \int_{\Omega} \gamma_{\epsilon}(z, \nabla u_\epsilon^{(m)}) |(u_\epsilon^{(m)})_{xx}|^2 \,dx +  C \left(1+ \int_{\Omega} |\nabla u_\epsilon^{(m)}|^{p(z)} \,dx \right)
\end{split}
\end{equation*}
with an arbitrary $\delta_3 \in (0,1)$ and $C$ depending upon $\delta_3$, $p$, $q$, $a$, $\partial \Omega$ and their differential properties, but not on $\epsilon$ and $m$. To complete the proof and obtain \eqref{eq:aux-est}, we gather the estimates of $J_1$, $J_2$, $J_a$, $J_{\partial \Omega}$ and choose $\delta_i$ so small that
\[
\min\{1,p^--1\}-\sum_{i=0}^3\delta_i=\eta>0.
\]
\end{proof}
\begin{Lem}
  \label{le:est-1}
Under the conditions of Lemma \ref{second}
  \begin{equation}
    \label{eq:ineq-0}
    \begin{split}
\sup_{(0,T)}\|\nabla
u^{(m)}_\epsilon{(\cdot, t)}\|^2_{2,\Omega} & + \int_{Q_T}
\gamma_{\epsilon}(z, \nabla u_\epsilon^{(m)})
|(u^{(m)}_\epsilon)_{xx}|^2 \,dz
\\
&
\leq  C {\rm e}^{C'T}\left(1+\|\nabla u_{0}\|_{2,\Omega}^{2}+
\|f_0\|_{L^{2}(0,T;W^{1,2}_0(\Omega))}^2\right)
\end{split}
\end{equation}
and
\begin{equation}
\label{eq:ineq-high}
\int_{Q_T}|\nabla u_{\epsilon}^{(m)}|^{q(z)}\,dz + \int_{Q_T}|\nabla u_{\epsilon}^{(m)}|^{p(z)+r}\,dz\leq C''\quad \text{for any $0<r<\dfrac{4p^-}{p^-(N+2)+2N}$}
\end{equation}
with constants $C$, $C'$, $C''$ independent of $m$ and $\epsilon$.
\end{Lem}

\begin{proof}
Multiplying \eqref{eq:aux-est} by ${\rm e}^{-2C_1 t}$ and simplifying, we obtain the following differential inequality:

\[
\dfrac{d}{dt}\left({\rm e}^{-2C_1t}\|\nabla u_{\epsilon}^{(m)}{(\cdot, t)}\|_{2,\Omega}^{2}\right)\leq C{\rm e}^{-2C_1 t}\left(1+\int_{\Omega}|\nabla u_{\epsilon}^{(m)}|^{p(z)}\,dx+\|f_0{(\cdot, t)}\|_{W^{1,2}_0(\Omega)}^2\right).
\]
Integrating it with respect to $t$ and taking into account \eqref{secderiboun}, \eqref{gradbound} we arrive at the following estimate: for every $t\in [0,T]$
\[
\begin{split}
\|\nabla u_{\epsilon}^{(m)}{(\cdot, t)}\|_{2,\Omega}^{2} & \leq C {\rm e}^{2C_1T}\left(\|\nabla u_0\|_{2,\Omega}^2+ {\rm e}^{T}\left(1+ \|u_0\|^{2}_{2,\Omega}+\|f_0\|_{2,Q_T}^2\right) +\|\nabla f\|_{2,Q_T}^{2}\right)
\\
& \leq C{\rm e}^{C'T}\left(1+\|u_0\|^2_{W^{1,2}_0(\Omega)}+\|f_0\|_{L^2(0,T;W^{1,2}_{0}(\Omega))}^2\right).
\end{split}
\]
Substitution of the above estimate into \eqref{eq:aux-est} gives
\[
\begin{split}
\frac{1}{2} \frac{d}{dt} \|\nabla
u^{(m)}_\epsilon{(\cdot, t)}\|^2_{2,\Omega} & + C_0 \int_{\Omega}
\gamma_{\epsilon}(z, \nabla u_\epsilon^{(m)})
|(u^{(m)}_\epsilon)_{xx}|^2 \,dx
\\
& \leq  C_1\left(1 + \int_{\Omega} |\nabla u_\epsilon^{(m)}|^{p(z)}\,dx  + \|\nabla u_{0}\|_{2,\Omega}^{2}+
\|f_0{(\cdot, t)}\|_{W^{1,2}_0(\Omega)}^2\right).
\end{split}
\]
Integrating it with respect to $t$ and using \eqref{gradbound} to estimate the integral of $|\nabla u_{\epsilon}^{(m)}|^{p(z)}$ on the right-hand side, we obtain
\begin{equation}
\notag
\label{eq:intermediate}
\int_{Q_T}
\gamma_{\epsilon}(z, \nabla u_\epsilon^{(m)})
|(u^{(m)}_\epsilon)_{xx}|^2 \,dz
\leq  C {\rm e}^{C'T}\left(1+\|\nabla u_{0}\|_{2,\Omega}^{2}+
\|f_0\|_{L^{2}(0,T;W^{1,2}_0(\Omega))}^2\right).
\end{equation}
To prove estimate \eqref{eq:ineq-high}, we make use of Theorem \ref{th:integr-par}. Let us fix a number $r\in (r_\ast,r^\ast)$ with $r_\ast$, $r^\ast$ defined in \eqref{eq:r}. Split the cylinder $Q_T$ into the two parts $Q^+_T=Q_T\cap \{p(z)+r\geq 2\}$, $Q^-_T=Q_T\cap \{p(z)+r<2\}$ and represent
\[
\int_{Q_T}|\nabla u_{\epsilon}^{(m)}|^{p(z)+r}\,dz=\int_{Q_T^+}|\nabla u_{\epsilon}^{(m)}|^{p(z)+r}\,dz + \int_{Q_T^-}|\nabla u_{\epsilon}^{(m)}|^{p(z)+r}\,dz\equiv I_++I_-.
\]
Since
\[
\begin{split}
I_+ & \leq \int_{Q_T^+}(\epsilon^2 + |\nabla u_{\epsilon}^{(m)}|^2)^{\frac{p(z)+r-2}{2}} |\nabla u_\epsilon^{(m)}|^2\,dz\leq \int_{Q_T} (\epsilon^2 + |\nabla u_{\epsilon}^{(m)}|^2)^{\frac{p(z)+r-2}{2}} |\nabla u_\epsilon^{(m)}|^2\,dz,
\end{split}
\]
the estimate on $I_+$ follows immediately from Theorem \ref{th:integr-par} and \eqref{eq:ineq-0}. To estimate $I_-$, we set $B_+=Q_T^-\cap \{z:|\nabla u_{\epsilon}^{(m)}|\geq \epsilon\}$, $B_-=Q_T^-\cap \{z:|\nabla u_{\epsilon}^{(m)}|< \epsilon\}$. The estimate on $I_-$ follows from Theorem \ref{th:integr-par} and \eqref{eq:ineq-0} because
\[
\begin{split}
I_- & =\int_{B_+\cup B_-}|\nabla u_{\epsilon}^{(m)}|^{p(z)+r}\,dz=\int_{B_+}(|\nabla u_{\epsilon}^{(m)}|^2)^{\frac{p(z)+r-2}{2}} |\nabla u_{\epsilon}^{(m)}|^{2}\,dz+\int_{B_-}\epsilon^{p(z)+r}\,dz
\\
& \leq 2^{\frac{2-r-p^-}{2}}\int_{B_+}(\epsilon^2 + |\nabla u_{\epsilon}^{(m)}|^2)^{\frac{p(z)+r-2}{2}} |\nabla u_{\epsilon}^{(m)}|^2\,dz+\epsilon^{p^-+r}T|\Omega|
\\
& \leq C\left(1+\int_{Q_T} (\epsilon^2 + |\nabla u_{\epsilon}^{(m)}|^2)^{\frac{p(z)+r-2}{2}} |\nabla u_{\epsilon}^{(m)}|^2\,dz\right).
\end{split}
\]
By combining the above estimates, using the Young inequality, and applying \eqref{eq:ineq-0}, \eqref{gradbound} and Theorem \ref{th:integr-par} we obtain \eqref{eq:ineq-high} with $r\in (r_\ast,r^\ast)$:

\[
\begin{split}
\int_{Q_T}|\nabla u_{\epsilon}^{(m)}|^{q(z)}\,dz  & + \int_{Q_T}|\nabla u_{\epsilon}^{(m)}|^{p(z)+r}\,dz\leq 1+ \int_{Q_T}|\nabla u_{\epsilon}^{(m)}|^{p(z)+r}\,dz
\\
& \leq C\left(1+\int_{Q_T}\gamma_{\epsilon}(z,\nabla u_\epsilon^{(m)})|(u_{\epsilon}^{(m)})_{xx}|^2\,dz+\int_{Q_T}|\nabla u_{\epsilon}^{(m)}|^{p(z)}\,dz\right)\leq C.
\end{split}
\]

\noindent If $r\in (0,r_\ast]$,  the required inequality follows from Young's inequality.
\end{proof}

\begin{remark}
\label{rem:eps-high-integr}
Under the conditions of Lemma \ref{le:est-1}
\begin{equation}
\label{eq:eps-high-integr}
 \int_{Q_T}(\epsilon^2+|\nabla u_{\epsilon}^{(m)}|^{2})^{\frac{p(z)+r}{2}}\,dz \leq C,\quad \epsilon\in (0,1),
\end{equation}
with an independent of $\epsilon$ and $m$ constant $C$.
\end{remark}

\begin{cor}
Let condition \eqref{eq:oscillation} be fulfilled. Under the conditions of Lemma \ref{le:est-1}
\begin{equation}
\label{eq:ineq-high-1}
\|(\epsilon^2+|\nabla u_\epsilon^{(m)}|^2)^{\frac{p(z)-2}{2}}\nabla u_\epsilon^{(m)}\|_{q'(\cdot),Q_T}\leq C
\end{equation}
with a constant $C$ independent of $m$ and $\epsilon$.
\end{cor}
\begin{proof}
Condition \eqref{eq:oscillation} entails the inequality
\[
\frac{q(z)(p(z)-1)}{q(z)-1}\leq q(z)\leq p(z)+r.
\]
By Young's inequality, the assertion follows then from \eqref{eq:eps-high-integr}:
\[
\int_{Q_T}(\epsilon^2+|\nabla u_\epsilon^{(m)}|^2)^{{\frac{q(z)(p(z)-1)}{2(q(z)-1)}}}\,dz\leq C\left(1+\int_{Q_T}|\nabla u_\epsilon^{(m)}|^{p(z)+r}\,dz\right)\leq C.
\]
\end{proof}
\begin{Lem}\label{le:time-der}
Assume that in the conditions of Lemma
\ref{second} {$u_0\in W^{1,2}_0(\Omega)\cap
\mathcal{V}_0(\Omega)$}. Then
\begin{equation}
\label{timederiest}
\begin{split}
\|(u^{(m)}_\epsilon)_t\|^2_{2,Q_T} &+\sup_{(0,T)}\int_{\Omega}\left((\epsilon^2+|\nabla u_{\epsilon}^{(m)}|^2)^{\frac{p(z)}{2}} + a(z)(\epsilon^2+|\nabla u_{\epsilon}^{(m)}|^2)^{\frac{q(z)}{2}}\right)\,dx \\
& \leq C\left(1+\int_\Omega\left(|\nabla u_0|^{p(x,0)}+a(x,0)|\nabla u_0|^{q(x,0)}\right)\,dx\right)+\|f_0\|_{2,Q_T}^2
\end{split}
\end{equation}
with an independent of $m$ and $\epsilon$ constant $C$, which depends on the constants in conditions \eqref{eq:Lip-p-q}.
\end{Lem}
\begin{proof}
By multiplying \eqref{system} with $(u^{(m)}_j)_t$ and summing over
$j=1,2,\dots,m$ we obtain the equality

\begin{equation}\label{est013}
\int_{\Omega} (u^{(m)}_\epsilon)^2_t \,dx + \int_{\Omega}
\gamma_{\epsilon}(z, \nabla u_\epsilon^{(m)})
\nabla u^{(m)}_\epsilon \cdot \nabla (u^{(m)}_\epsilon)_t \,dx =
\int_{\Omega} f_0 (u^{(m)}_\epsilon)_t \,dx.
\end{equation}
Using the identity
\[
\begin{split}
a(z) &(\epsilon^2 + |\nabla u^{(m)}_\epsilon |^2)^{\frac{q(z)-2}{2}}
\nabla u^{(m)}_\epsilon \cdot \nabla (u^{(m)}_\epsilon)_t =
\frac{d}{dt} \left( \frac{ a(z)(\epsilon^2 + |\nabla u^{(m)}_\epsilon
|^2)^{\frac{q(z)}{2}}}{q(z)}\right)
\\ & + \frac{a(z) q_t(z) (\epsilon^2 + |\nabla u^{(m)}_\epsilon
|^2)^{\frac{q(z)}{2}}}{{q^2(z)}} \left(1- \frac{q(z)}{2}
\ln((\epsilon^2 + |\nabla u^{(m)}_\epsilon |^2))\right) -  \frac{ a_t (\epsilon^2 + |\nabla u_\epsilon^{(m)}|^2)^\frac{q(z)}{2}}{q(z)}
\end{split}
\]
we rewrite \eqref{est013} as
\begin{equation}\label{est014}
\begin{split}
\|(u^{(m)}_\epsilon)_t(\cdot,t)\|^2_{2,\Omega} & + \frac{d}{dt}\int_{\Omega} \left(\frac{(\epsilon^2 + |\nabla u^{(m)}_\epsilon |^2)^{\frac{p(z)}{2}}}{p(z)} +  \frac{a(z) (\epsilon^2 + |\nabla u^{(m)}_\epsilon |^2)^{\frac{q(z)}{2}}}{q(z)}\right)  \,dx\\
&=\int_{\Omega} f_0 (u^{(m)}_\epsilon)_t \,dx  - \int_{\Omega} \frac{p_t (\epsilon^2 + |\nabla u^{(m)}_\epsilon |^2)^{\frac{p(z)}{2}}}{p^2(z)} \left(1- \frac{p(z)}{2} \ln(\epsilon^2 + |\nabla u^{(m)}_\epsilon |^2)\right)\,dx\\
& \qquad  - \int_{\Omega} \frac{a(z) q_t(z) (\epsilon^2 + |\nabla u^{(m)}_\epsilon
|^2)^{\frac{q(z)}{2}}}{q^2(z)} \left(1- \frac{q(z)}{2}
\ln((\epsilon^2 + |\nabla u^{(m)}_\epsilon |^2))\right)\,dx
\\
& \qquad + \int_{\Omega} \frac{a_t (\epsilon^2 + |\nabla u_\epsilon^{(m)}|^2)^\frac{q(z)}{2}}{q(z)} \,dx
\\
&
\equiv \int_{\Omega} f _0(u^{(m)}_\epsilon)_t \,dx +\mathcal{J}_1+\mathcal{J}_2+\mathcal{J}_3.
\end{split}
\end{equation}
The first term on the right-hand side of \eqref{est014} is estimated by the Cauchy inequality:
\begin{equation}\label{est016}
\left|\int_{\Omega} f_0 (u^{(m)}_\epsilon)_t \,dx\right| \leq
\frac{1}{2} \|(u^{(m)}_\epsilon)_t {(\cdot,t)}\|^2_{2,\Omega} +
\frac{1}{2}\|f_0{(\cdot,t)}\|^2_{2,\Omega}.
\end{equation}
To estimate $\mathcal{J}_i$ we use \eqref{secderiboun}, \eqref{gradbound}, \eqref{ineqq}, \eqref{eq:logest1} and \eqref{eq:ineq-high}. Fix two numbers $r_1 \in (r_\ast, r^\ast), r_2 \in (0,1)$ such that
\[
q(z) + r_2 < p(z) + r_1 < p(z) + r^\ast.
\]
Then
\begin{equation*}
\begin{split}
\sum_{i=1}^{3}|\mathcal{J}_i| & \leq
C_1 \left(1+ \int_{\Omega} |\nabla u_\epsilon^{(m)}|^{p(z)}
\,dx + \int_{\Omega} |\nabla u_\epsilon^{(m)}|^{q(z)}
\,dx\right)
\\
& \qquad + C_2 \int_{\Omega} (\epsilon^2 + |\nabla
u^{(m)}_\epsilon |^2)^{\frac{p(z)}{2}} \ln(\epsilon^2 + |\nabla
u^{(m)}_\epsilon |^2) \,dx
\\
& \qquad
+ C_3 \int_{\Omega} (\epsilon^2 + |\nabla
u^{(m)}_\epsilon |^2)^{\frac{q(z)}{2}} \ln(\epsilon^2 + |\nabla
u^{(m)}_\epsilon |^2) \,dx \\
& \leq C_4\left(1+ \int_{\Omega} (\epsilon^2 + |\nabla
u^{(m)}_\epsilon |^2)^{\frac{p(z)+ r_1}{2}} \,dx\right).
\end{split}
\end{equation*}
The required inequality \eqref{timederiest} follows after
gathering the above estimates, integrating the result in $t$
and applying {\eqref{eq:choice}}.
\end{proof}

\subsection{A priori estimates II: the case $b\not\equiv 0$}
We proceed to derive a priori estimates in the case when the equation contains the nonlinear source. The difference in the arguments consists in the necessity to estimate the integrals of the terms $b|u_\epsilon^{(m)}|^{\sigma(z)}$,  $b|u_\epsilon^{(m)}|^{\sigma(z)-2}u_{\epsilon}^{(m)}\Delta u_{\epsilon}^{(m)}$, $b|u_\epsilon^{(m)}|^{\sigma(z)-2}u_{\epsilon}^{(m)} u_{\epsilon t}^{(m)}$.

\medskip
1) Let us multiply $j^{\text{}th}$ equation of \eqref{system} by $u_j^{(m)}$ and sum up. In the result we arrive at equality \eqref{est01} with the right-hand side containing the additional term
\[
\mathcal{I}_0\equiv \int_{\Omega}b(z)|u_\epsilon^{(m)}|^{\sigma(z)}\,dx.
\]
Let $2(\sigma^+-1)<p^-$. Using the inequalities of Young and Poincar\'{e} we find that for every $t\in (0,T)$
\[
\begin{split}
|\mathcal{I}_0| & \leq B\left(1+\int_{\Omega}|u_{\epsilon}^{(m)}|^{2(\sigma^+-1)}\,dx +\int_{\Omega}|u_\epsilon^{(m)}|^2\,dx\right)
\\
& \leq C_\delta+\delta\int_{\Omega}|u_{\epsilon}^{(m)}|^{p^-}\,dx + \int_{\Omega}|u_\epsilon^{(m)}|^2\,dx
\\
& \leq C'_\delta+\widehat{C}\delta\int_{\Omega}|\nabla u_{\epsilon}^{(m)}|^{p^-}\,dx +C\int_{\Omega}|u_\epsilon^{(m)}|^2\,dx
\\
& \leq C''_\delta+\widehat{C}\delta\int_{\Omega}|\nabla u_{\epsilon}^{(m)}|^{p(z)}\,dx +C\int_{\Omega}|u_\epsilon^{(m)}|^2\,dx
\end{split}
\]
where $\delta\in (0,1)$ is an arbitrary constant and $\widehat{C}$ is the constant from inequality \eqref{eq:Poincare} with $r=p^-$. We plug this estimate into \eqref{eq:option-1} and use \eqref{bdd1} with $a\equiv 1$ and $q$ substituted by $p$. Chosing $\delta$ sufficiently small, we transform \eqref{eq:option-1} to the form
\[
\frac{1}{2} \frac{d}{dt} \|u^{(m)}_\epsilon{(\cdot,t)}\|^2_{2,\Omega} + (1-C\delta)\int_{\Omega} \gamma_{\epsilon}(z, \nabla u_\epsilon^{(m)}) |\nabla u^{(m)}_\epsilon|^2 ~dx \leq  C'\left(1+\|f_0{(\cdot,t)}\|^2_{2,\Omega} + \|u^{(m)}_\epsilon{(\cdot,t)}\|^2_{2,\Omega}\right).
\]
Integrating this inequality in $t$ we obtain the following counterpart  of Lemma \ref{1st}.
\begin{Lem}
\label{le:source-1}
Assume that $a{(\cdot)}$, $p{(\cdot)}$, $q{(\cdot)}$, $u_0$, $f_0$ satisfy the conditions of Lemma \ref{1st}. If $\sigma,b$ are measurable and bounded functions in ${Q}_T$ and $1<\sigma^-\leq \sigma^+<1+\dfrac{p^-}{2}$, then
\begin{equation}\label{secderiboun-new}
\begin{split}
\sup_{t \in (0,T)} \|u^{(m)}_\epsilon{(\cdot,t)}\|^2_{2,\Omega} + \int_{Q_T} \gamma_{\epsilon}(z, \nabla u_\epsilon^{(m)}) |\nabla u^{(m)}_\epsilon|^2 ~dz \leq C_1 {\rm e}^{ T} (\|f_0\|^2_{2,Q_T}  + \|u_0\|^2_{2,\Omega})+C_0
\end{split}
\end{equation}
and
\begin{equation}\label{gradbound-new}
\begin{split}
\int_{Q_{T}} \left(|\nabla u^{(m)}_\epsilon|^{p(z)} + a(z) |\nabla u^{(m)}_\epsilon|^{q(z)}\right) ~dx ~dt \leq C_2
\int_{Q_T}  \gamma_{\epsilon}(z, \nabla u_\epsilon^{(m)}) |\nabla u^{(m)}_\epsilon|^2 ~dz + C_3
\end{split}
\end{equation}
with independent of $\epsilon$ and $m$ constants $C_i$.
\end{Lem}
\medskip
2) Estimate on $\|\nabla u_{\epsilon}^{(m)}(t)\|_{2,\Omega}$. We follow the proof of Lemma \ref{second}: multiplying each of equations in \eqref{system} by {$\lambda_j u_{j}^{(m)}$} and summing the results we arrive at equality \eqref{eq:prime} with the additional term in the right-hand side. The new term can be transformed by means of integration by parts in $\Omega$:
\[
\begin{split}
\mathcal{I}_1 & =\int_{\Omega}b(z)|u_{\epsilon}^{(m)}|^{\sigma(z)-2}u_{\epsilon}^{(m)} \Delta u_{\epsilon}^{(m)}\,dx
\\
&
\leq \int_{\Omega}(\sigma(z)-1)|b(z)||u_{\epsilon}^{(m)}|^{\sigma(z)-2}|\nabla u_{\epsilon}^{(m)}|^2\,dx
\\
&
\qquad + \int_{\Omega}|u_{\epsilon}^{(m)}|^{\sigma(z)-1} |\nabla b||\nabla u_{\epsilon}^{(m)}|\,dx + \int_{\Omega} |b(z)| |u_{\epsilon}^{(m)}|^{\sigma(z)-1}|\ln ||u_{\epsilon}^{(m)}|||\nabla u_{\epsilon}^{(m)}||\nabla \sigma|\,dx
\\
& \equiv \mathcal{K}_1+ \mathcal{K}_2 + \mathcal{K}_3.
\end{split}
\]
To estimate $\mathcal{K}_3$ we assume that the functions $|b|$ and $|\nabla \sigma|$ are bounded a.e. in $Q_T$ and then apply the Cauchy inequality, \eqref{ineqq}, and the Poincar\'{e} inequality: if $2(\sigma^+-1)<p^-$, there exists a constant $\mu>0$ such that $2(\sigma^+-1)+\mu\leq p^-$
\[
\begin{split}
\mathcal{K}_3 & \leq C\left(1+\|\nabla u_{\epsilon}^{(m)}{(\cdot,t)}\|_{2,\Omega}^{2} + \int_{\Omega}|u_{\epsilon}^{(m)}|^{2(\sigma(z)-1+\mu)}\,dx\right)
\\
&
\leq C\left(1+\|\nabla u_{\epsilon}^{(m)}{(\cdot,t)}\|_{2,\Omega}^{2} + \int_{\Omega}|u_{\epsilon}^{(m)}|^{2(\sigma^+-1)+\mu}\,dx\right)
\\
& \leq  C'\left(1+\|\nabla u_{\epsilon}^{(m)}{(\cdot,t)}\|_{2,\Omega}^{2} + \int_{\Omega}|\nabla u_{\epsilon}^{(m)}|^{2(\sigma^+-1)+\mu}\,dx\right)
\\
&
\leq C''\left(1+\|\nabla u_{\epsilon}^{(m)}{(\cdot,t)}\|_{2,\Omega}^{2} + \int_{\Omega}|\nabla u_{\epsilon}^{(m)}|^{p(z)}\,dx\right).
\end{split}
\]
$\mathcal{K}_2$ is estimated likewise: if $|\nabla b|$ is bounded a.e. in $Q_T$ and $2(\sigma^+-1)<p^-$, then
\[
\mathcal{K}_2\leq C\left(1+\|\nabla u_{\epsilon}^{(m)}{(\cdot,t)}\|_{2,\Omega}^{2} + \int_{\Omega}|u_{\epsilon}^{(m)}|^{2(\sigma(z)-1)}\,dx\right) \leq C'\left(1+\|\nabla u_{\epsilon}^{(m)}{(\cdot,t)}\|_{2,\Omega}^{2} + \int_{\Omega}|\nabla u_{\epsilon}^{(m)}|^{p(z)}\,dx\right).
\]
To estimate $\mathcal{K}_1$ we assume that $\sigma^-\geq 2$ and notice that the restriction on $p^-$ and $\sigma^+$ imposed to estimate $\mathcal{K}_2$ and $\mathcal{K}_3$ yields
\[
4\leq 2\sigma^-\leq 2\sigma^+<2+p^-\quad \Rightarrow\quad p^->2\quad \Rightarrow\quad \sigma^+<1+\dfrac{p^-}{2}<p^-.
\]
Using this observation and the Young inequality we estimate $\mathcal{K}_1$ as follows:
\[
\begin{split}
\mathcal{K}_1 & \leq C\left(\int_{\Omega}|\nabla u_{\epsilon}^{(m)}|^{p(z)}\,dx +\int_{\Omega}|u_{\epsilon}^{(m)}|^{p(z)\frac{\sigma(z)-2}{p(z)-2}}\,dx\right)
\\
&
\leq C \left(1+\int_{\Omega}|\nabla u_{\epsilon}^{(m)}|^{p(z)}\,dx +\int_{\Omega}|u_{\epsilon}^{(m)}|^{p(z)\frac{\sigma^+-2}{p^--2}}\,dx\right)
\\
&
\leq C' \left(1+\int_{\Omega}|\nabla u_{\epsilon}^{(m)}|^{p(z)}\,dx +\int_{\Omega}|u_{\epsilon}^{(m)}|^{p(z)}\,dx\right).
\end{split}
\]
Following the proof of Lemma \ref{le:est-1} and taking into account the estimates on $\mathcal{K}_i$ we arrive at the inequality
\begin{equation}
    \label{eq:ineq-0-prim}
    \begin{split}
\sup_{(0,T)}\|\nabla
u^{(m)}_\epsilon{(\cdot,t)}\|^2_{2,\Omega} & + \int_{Q_T}
\gamma_{\epsilon}(z, \nabla u_\epsilon^{(m)})
|(u^{(m)}_\epsilon)_{xx}|^2 \,dz
\\
&
\leq  C {\rm e}^{C'T}\left(1+\|\nabla u_{0}\|_{2,\Omega}^{2}+
\|f_0\|_{L^{2}(0,T;W^{1,2}_0(\Omega))}^2\right)
\\
& + C''{\rm e}^{C'T}\left(\int_{Q_T}|\nabla u_{\epsilon}^{(m)}|^{p(z)}\,dz+ \int_{Q_T}|u_{\epsilon}^{(m)}|^{p(z)}\,dz\right)
\end{split}
\end{equation}
with new constants $C$, $C'$, $C''$ which do not depend on $\epsilon$ and $m$. The last term on the right-hand side of this inequality is estimated by virtue of Lemma \ref{le:embed-A} and estimates \eqref{secderiboun-new}, \eqref{gradbound-new}.

\begin{Lem}
  \label{le:source-2}
 Let in the conditions of Lemma \ref{le:source-1},  $2\leq \sigma^-\leq \sigma^+<1+\dfrac{p^-}{2}$ holds. If $\|\nabla b\|_{\infty,Q_T}<\infty$ and $\|\nabla \sigma\|_{\infty,Q_T}<\infty$, then
  \begin{equation}
    \label{eq:ineq-0-source}
    \begin{split}
\sup_{(0,T)}\|\nabla
u^{(m)}_\epsilon{(\cdot,t)}\|^2_{2,\Omega} & + \int_{Q_T}
\gamma_{\epsilon}(z, \nabla u_\epsilon^{(m)})
|(u^{(m)}_\epsilon)_{xx}|^2 \,dz
\\
&
\leq  C {\rm e}^{C'T}\left(\widehat{C}+\|u_{0}\|_{W^{1,2}_{0}(\Omega)}^{2}+
\|f_0\|_{L^{2}(0,T;W^{1,2}_0(\Omega))}^2\right)
\end{split}
\end{equation}
with an independent of $\epsilon$ and $m$ constants $C$, $C'$, and a constant $\widehat{C}$ depending only on $T$, and the quantities on the right-hand sides of \eqref{secderiboun}, \eqref{gradbound}.
\end{Lem}
\medskip
3) Estimate on $\|u_{\epsilon t}^{(m)}\|_{2,Q_T}$. We follow the proof of Lemma \ref{le:time-der}. Multiplying \eqref{system} by $(u_\epsilon^{(m)})_t$ and summing the results we obtain equality \eqref{est014} with the additional term on the right-hand side:
\[
\mathcal{M}_0\equiv \int_{\Omega}b(z)|u_\epsilon^{(m)}|^{\sigma(z)-2}u_\epsilon^{(m)} (u_\epsilon^{(m)})_t\,dx.
\]
By Young's inequality
\[
\mathcal{M}_0\leq C \int_{\Omega}|u_\epsilon^{(m)}|^{2(\sigma(z)-1)}\,dx +\frac{1}{2}\int_{\Omega}(u_\epsilon^{(m)})^2_t\,dx.
\]
Combining this inequality with \eqref{timederiest} and taking into account the inequality $2(\sigma(z)-1)<p(z)$ following from the inequality $2(\sigma^+-1)<p^-$,  we obtain
\begin{equation}
\label{timederiest-source}
\begin{split}
\frac{1}{2}\|(u^{(m)}_\epsilon)_t\|^2_{2,Q_T} &+\sup_{(0,T)}\int_{\Omega}\left((\epsilon^2+|\nabla u_{\epsilon}^{(m)}|^2)^{\frac{p(z)}{2}} + a(z)(\epsilon^2+|\nabla u_{\epsilon}^{(m)}|^2)^{\frac{q(z)}{2}}\right)\,dx \\
& \leq C\left(1+\int_\Omega\left(|\nabla u_0|^{p(x,0)}+a(x,0)|\nabla u_0|^{q(x,0)}\right)\,dx\right)+\|f_0\|_{2,Q_T}^2
\\
&
+C'\left(1+\int_{Q_T}|u_{\epsilon}|^{p(z)}\,dz\right).
\end{split}
\end{equation}
The last integral on the right-hand side is estimated by virtue of Lemma \ref{le:embed-A} and the estimates of Lemma \ref{le:source-1}.

\begin{Lem}
\label{le:source-3}
Let the conditions of Lemma \ref{le:source-2} be fulfilled. Then
\begin{equation}
\label{eq:timederiest-source}
\begin{split}
\frac{1}{2}\|(u^{(m)}_\epsilon)_t\|^2_{2,Q_T} &+\sup_{(0,T)}\int_{\Omega}\left((\epsilon^2+|\nabla u_{\epsilon}^{(m)}|^2)^{\frac{p(z)}{2}} + a(z)(\epsilon^2+|\nabla u_{\epsilon}^{(m)}|^2)^{\frac{q(z)}{2}}\right)\,dx \\
& \leq C\left(1+\int_\Omega\left(|\nabla u_0|^{p(x,0)}+a(x,0)|\nabla u_0|^{q(x,0)}\right)\,dx\right)+\|f_0\|_{2,Q_T}^2
+C'
\end{split}
\end{equation}
with constants $C$, $C'$ independent of $\epsilon$ and $m$.
\end{Lem}

\section{Existence and uniqueness of strong solution}
\label{sec:reg-existence}
In this section, we prove that the regularized problem \eqref{eq:reg-prob} and the degenerate problem \eqref{eq:main} have strong solutions and derive conditions of uniqueness of these solutions.
\subsection{Regularized problem}
\begin{thm}
\label{th:exist-reg} Let $u_0$, $f$, $p$, $q$, $a$ and $\partial\Omega$
satisfy the conditions of Theorem \ref{PPresu2}. Then for every
$\epsilon\in (0,1)$ problem \eqref{eq:reg-prob} has a unique  solution $u_\epsilon$ which satisfies the estimates
\begin{equation}
\label{eq:est-strong-eps-1}
\begin{split}
& \|u_\epsilon\|_{{ \mathcal{W}_{q(\cdot)}(Q_T)}}\leq C_0,
\\
&
\operatorname{ess}\sup_{(0,T)}\|u_\epsilon{(\cdot, t)}\|_{2,\Omega}^{2}
+
\|u_{\epsilon t}\|_{2,Q_T}^{2}+ \operatorname{ess}\sup_{(0,T)}\|\nabla u_\epsilon{(\cdot, t)}\|_{2,\Omega}^{2}
\\
&\qquad \qquad
+
\operatorname{ess}\sup_{(0,T)}\int_{\Omega}
\left((\epsilon^2+|\nabla u_\epsilon^{(m)}|^2)^{\frac{p(z)}{2}}+a(z)(\epsilon^2+|\nabla u_\epsilon^{(m)}|^2)^{\frac{q(z)}{2}}\right)
dx\leq C_0
\end{split}
\end{equation}
with a constant $C_0$ depending on the data but not on $\epsilon$. Moreover, $u_\epsilon$ possesses the property of global higher
integrability of the gradient: for every
\[
\delta\in (0,r^\ast), \qquad r^\ast=\dfrac{4p^-}{p^-(N+2)+2N},
\]
there exists a constant $C=C\left(\partial\Omega,
N,p^\pm,\delta,\|u_0\|_{W^{1,2}_0(\Omega)},\|f\|_{L^{2}(0,T;W^{1,2}_{0}(\Omega))}\right)$
such that
\begin{equation}
\label{eq:grad-high-eps} \int_{Q_T}|\nabla
u_{\epsilon}|^{p(z)+\delta}\,dz\leq C.
\end{equation}
\end{thm}
\begin{proof}
Let $\epsilon\in (0,1)$ be a fixed parameter. Under the assumptions of Theorem \ref{PPresu2},
there exists a sequence of Galerkin approximations $u^{(m)}_\epsilon$ defined by formulas \eqref{eq:coeff} which satisfies estimates
\eqref{secderiboun}, \eqref{gradbound}, \eqref{eq:ineq-0},
\eqref{eq:ineq-high}, \eqref{eq:ineq-high-1} and \eqref{timederiest}. These uniform in $m$
and $\epsilon$ estimates enable one to extract a subsequence
$u^{(m)}_{\epsilon}$ (for which we keep the same name), and
functions $u_\epsilon$, $\eta_\epsilon$, $\chi_\epsilon$ such that
\begin{equation}
\label{eq:conv}
\begin{split}
& u^{(m)}_\epsilon \to u_\epsilon \quad \text{$\star$-weakly in $
L^\infty(0,T;L^2(\Omega))$}, \quad \text{$(u^{(m)}_{\epsilon})_t \rightharpoonup (u_{\epsilon})_t$ in
$L^2(Q_T)$},
\\
& \text{$\nabla u^{(m)}_\epsilon \rightharpoonup \nabla
u_\epsilon$  in $(L^{p(\cdot)}(Q_T))^N$}, \ \quad \text{$\nabla u^{(m)}_\epsilon \rightharpoonup \nabla
u_\epsilon$  in $(L^{q(\cdot)}(Q_T))^N$},
\\
& \text{$(\epsilon^2 + |\nabla
u^{(m)}_\epsilon|^2)^{\frac{p(z)-2}{2}} \nabla
u^{(m)}_\epsilon\rightharpoonup \eta_\epsilon$ in
$(L^{q'(\cdot)}(Q_T))^N$,}
\\
& \text{$(\epsilon^2 + |\nabla
u^{(m)}_\epsilon|^2)^{\frac{q(z)-2}{2}} \nabla
u^{(m)}_\epsilon\rightharpoonup {\chi}_\epsilon$ in
$(L^{q'(\cdot)}(Q_T))^N$.}
\end{split}
\end{equation}
{In the third line we make use of the uniform estimate
\[
\int_{Q_T}(\epsilon^2+|\nabla u_{\epsilon}^{(m)}|^2)^{\frac{q(z)(p(z)-1)}{2(q(z)-1)}}\,dz\leq C\left(1+\int_{Q_T}|\nabla u_{\epsilon}^{(m)}|^{p(z)+r}\,dz\right)\leq C,
\]
which follows from \eqref{eq:oscillation} and \eqref{eq:ineq-high}.}
The functions $u^{(m)}_\epsilon$ and $(u^{(m)}_\epsilon)_t$ are uniformly
bounded in $L^{\infty}(0,T;W_0^{1,p^-}(\Omega))$ and
$L^2(0,T; L^2(\Omega))$ respectively, and $W_0^{1,q(\cdot,t)}(\Omega) \subseteq
W^{1,q^-}_0(\Omega)\hookrightarrow L^2(\Omega)$. By \cite[Sec.8, Corollary 4]{simon-1987} the sequence $\{u^{(m)}_\epsilon\}$ is relatively compact in $C([0,T];L^2(\Omega))$, {\it i.e.}, there exists a subsequence $\{u^{(m_k)}_\epsilon\}$, which we assume coinciding with $\{u^{(m)}_{\epsilon}\}$, such that $u^{(m)}_\epsilon \to u_\epsilon$ in
$C([0,T];L^2(\Omega))$ and a.e. in $Q_T.$
Let us define
\[
\mathcal{P}_m
=\left\{\phi:\,\phi=\sum_{i=1}^{m}\psi_i(t)\phi_{i}(x),\,\text{$\psi_i$ are absolutely continuous in $[0,T]$} \right\}.
\]
Fix some $m\in \mathbb{N}$. By the method of construction
$u^{(m)}_\epsilon\in \mathcal{P}_m$. Since $\mathcal{P}_{k}\subset
\mathcal{P}_{m}$ for $k<m$, then for every $\xi_k\in
\mathcal{P}_k$ with $k\leq m$
\begin{equation}
\label{eq:ident-m} \int_{Q_T} u^{(m)}_{\epsilon t} \xi_k \,dz +
\int_{Q_T} \gamma_{\epsilon}(z, \nabla u_\epsilon^{(m)}) \nabla u^{(m)}_\epsilon
\cdot \nabla \xi_k \,dz = \int_{Q_T} f_0 \xi_k \,dz.
\end{equation}
Let {$\xi\in \mathcal{W}_{q(\cdot)}(Q_T)$. The space $C^{\infty}([0,T];C_{0}^{\infty}(\Omega))$ is dense in $\mathcal{W}_{q(\cdot)}(Q_T)$, therefore there exists a sequence $\{\xi_k\}$ such that $\xi_k\in \mathcal{P}_k$ and $\xi_k\to \xi\in \mathcal{W}_{q(\cdot)}(Q_T)$}. {
If $U_m\rightharpoonup U$ in $L^{q'(\cdot)}(Q_T)$, then for every $V\in L^{q(\cdot)}(Q_T)$ we have
\[
a(z)V\in L^{q(\cdot)}(Q_T) \quad  \text{and}\quad \displaystyle\int_{Q_T}aU_mV\,dz\to \int_{Q_T}aUV\,dz.
\]
Using this fact we pass to the limit as $m\to\infty$ in \eqref{eq:ident-m} with a fixed $k$}, and then letting $k\to \infty$, we conclude that

\begin{equation}
\label{eq:ident-lim} \int_{Q_T} u_{\epsilon t} \xi \,dz +
\int_{Q_T} \eta_\epsilon \cdot \nabla \xi \,dz + \int_{Q_T} a(z) \ {\chi}_\epsilon \cdot \nabla \xi \,dz = \int_{Q_T} f_0 \xi
\,dz
\end{equation}
for all $\xi \in {\mathcal{W}_{q(\cdot)}(Q_T)}$. To identify the limit vectors $\eta_{\epsilon}$ and ${\chi}_\epsilon$ we use the
classical argument based on monotonicity. The flux function $\gamma_{\epsilon}(z, \nabla u_\epsilon^{(m)})\nabla u_\epsilon^{(m)}$ is monotone:
\begin{equation}
\label{eq:mon}
(\gamma_{\epsilon}(z,\xi) \xi- \gamma_{\epsilon}(z,\zeta)\zeta, \xi-\zeta) \geq 0 \quad \text{for all $\xi, \zeta \in \mathbb{R}^N$, $z \in Q_T$, $\epsilon >0$},
\end{equation}
see, e.g., \cite[Lemma 6.1]{A-S} for the proof. By virtue of \eqref{eq:mon}, for every $\psi\in \mathcal{P}_m$
\begin{equation}
\label{eq:mon-calc}
\begin{split}
\gamma_\epsilon(z,\nabla u_\epsilon^{(m)})|\nabla u_{\epsilon}^{(m)}|^2 & = \gamma_\epsilon(z,\nabla u_\epsilon^{(m)}) \nabla u_\epsilon^{(m)}\cdot(\nabla u_{\epsilon}^{(m)}-\nabla \psi) + \gamma_\epsilon(z,\nabla u_\epsilon^{(m)})\nabla u_\epsilon^{(m)}\cdot \nabla \psi
\\
& = (\gamma_\epsilon(z,\nabla u_\epsilon^{(m)}) \nabla u_\epsilon^{(m)}-\gamma_{\epsilon}(z,\nabla \psi)\nabla \psi)\cdot(\nabla u_{\epsilon}^{(m)}-\nabla \psi)
\\
&
\qquad + \gamma_{\epsilon}(z,\nabla \psi)\nabla \psi\cdot(\nabla u_{\epsilon}^{(m)}-\nabla \psi)
+ \gamma_\epsilon(z,\nabla u_\epsilon^{(m)})\nabla u_\epsilon^{(m)}\cdot \nabla \psi
\\
& \geq \gamma_{\epsilon}(z,\nabla \psi)\nabla \psi\cdot(\nabla u_{\epsilon}^{(m)}-\nabla \psi)
+ \gamma_\epsilon(z,\nabla u_\epsilon^{(m)})\nabla u_\epsilon^{(m)}\cdot \nabla \psi.
\end{split}
\end{equation}
By taking $\xi_k=u^{(m)}_\epsilon$ in \eqref{eq:ident-m} we obtain: for every $\psi\in \mathcal{P}_k$
with $k\leq m$
\[
\begin{split}
0 & =
\int_{Q_T} (u^{(m)}_\epsilon)_tu^{(m)}_\epsilon \,dz
+ \int_{Q_T} \gamma_\epsilon(z,\nabla u_\epsilon^{(m)})|\nabla u_{\epsilon}^{(m)}|^2\,dz -\int_{Q_T} f_0 u^{(m)}_{\epsilon} \,dz
\\
& \geq \int_{Q_T} (u^{(m)}_\epsilon)_tu^{(m)}_\epsilon \,dz
+\int_{Q_T}\gamma_{\epsilon}(z,\nabla \psi)\nabla \psi\cdot \nabla
(u^{(m)}_{\epsilon}-\psi)\,dz\\
& \quad + \int_{Q_T} \gamma_{\epsilon}(z,\nabla u_\epsilon^{(m)})\nabla u_\epsilon^{(m)}\cdot \nabla \psi\,dz
-\int_{Q_T} f_0
u^{(m)}_{\epsilon} \,dz.
\end{split}
\]
{Notice that $(u_{\epsilon}^{(m)},(u_{\epsilon}^{(m)})_t)_{2,Q_T}\to (u_{\epsilon t},u_\epsilon)_{2,Q_T}$ as $m\to \infty$ as the product of weakly and strongly convergent sequences}. This fact together with  \eqref{eq:conv} means that each term of the last inequality has a limit as $m\to\infty$. Letting $m\to \infty$ and using \eqref{eq:ident-lim}, we find that  for every $\psi\in \mathcal{P}_{k}$
\[
\begin{split}
0 & \geq \int_{Q_T}u_{\epsilon}u_{\epsilon t}\,dz+ \int_{Q_T}
\gamma_{\epsilon}(z,\nabla \psi)\nabla \psi\cdot
\nabla (u_{\epsilon}-\psi)\,dz
+ \int_{Q_T}(\eta_{\epsilon}+a(z)\chi_{\epsilon})\cdot\nabla \psi\,dz
-\int_{Q_T} f_0 u_{\epsilon} \,dz
\\
& = \int_{Q_T}\left((\epsilon^2+|\nabla
\psi|^{2})^{\frac{p(z)-2}{2}}\nabla
\psi-\eta_{\epsilon}\right)\cdot \nabla (u_\epsilon-\psi)\,dz \\
& \qquad + \int_{Q_T} a(z) \left((\epsilon^2+|\nabla
\psi|^{2})^{\frac{q(z)-2}{2}}\nabla
\psi-{\chi}_{\epsilon}\right)\cdot \nabla (u_\epsilon-\psi)\,dz.
\end{split}
\]
By the density of $\bigcup_{k=1}^\infty\mathcal{P}_k$  in {$\mathcal{W}_{q(\cdot)}(Q_T)$},
the last inequality also holds for every $\psi\in {\mathcal{W}_{q(\cdot)}(Q_T)}$. Take $\psi=u_\epsilon+\lambda \xi$ with a constant $\lambda>0$ and an
arbitrary $\xi\in {\mathcal{W}_{q(\cdot)}(Q_T)}$. Then
\begin{equation}
\begin{split}
\lambda
\bigg[\int_{Q_T} &\left((\epsilon^2+|\nabla(u_\epsilon+\lambda
\xi)|^{2})^{\frac{p(z)-2}{2}}\nabla (u_\epsilon+\lambda
\xi)-\eta_{\epsilon}\right)\cdot \nabla \xi\,dz \\
& \quad + \int_{Q_T}a(z) \left((\epsilon^2+|\nabla(u_\epsilon+\lambda
\xi)|^{2})^{\frac{q(z)-2}{2}}\nabla (u_\epsilon+\lambda
\xi)-{\chi}_{\epsilon}\right)\cdot \nabla \xi\,dz\bigg] \leq 0.
\end{split}
\end{equation}
Simplifying and letting $\lambda\to 0$ we find that
\[
\int_{Q_T}\left(\gamma_{\epsilon}(z,\nabla u_{\epsilon})\nabla u_{\epsilon} - (\eta_{\epsilon}+ a(z) {\chi}_{\epsilon})\right)\cdot \nabla \xi\,dz\leq 0\quad \forall \xi\in {\mathcal{W}_{q(\cdot)}(Q_T)},
\]
which is possible only if
\[
\int_{Q_T}\left(\gamma_{\epsilon}(z,\nabla u_{\epsilon})\nabla u_{\epsilon} - (\eta_{\epsilon}+ a(z) {\chi}_{\epsilon})\right)\cdot \nabla \xi\,dz= 0\quad \forall \xi\in {\mathcal{W}_{q(\cdot)}(Q_T)},
\]
The initial condition for $u_\epsilon$ is fulfilled by continuity
because $u_\epsilon\in C([0,T];L^{2}(\Omega))$.

Uniqueness of the weak solution is an immediate byproduct of monotonicity. Let $u,v$ are two solutions of problem \eqref{eq:reg-prob}. Take an arbitrary $\tau\in (0,T]$. Choosing $u-v$ for the test function in equalities \eqref{eq:def} for $u$ and $v$ in the cylinder $Q_\tau=\Omega\times (0,\tau)$, subtracting the results and applying \eqref{eq:mon} we arrive at the inequality

\[
\frac{1}{2}\|u-v\|_{2,\Omega}^{2}(\tau) =\int_{Q_\tau}(u-v)(u-v)_t\,dz\leq 0.
\]
It follows that $u(x,\tau)=v(x,\tau)$ a.e. in $\Omega$ for every $\tau\in [0,T]$.

Estimates \eqref{eq:est-strong-eps-1} follow from the uniform in $m$ estimates on the functions $u_{\epsilon}^{(m)}$ and their derivatives, the properties of weak convergence \eqref{eq:conv} and lower semicontinuity of the modular. Inequality \eqref{eq:eps-high-integr} yields that for every $\delta\in (0,r^\ast)$ the sequence $\{\nabla u_{\epsilon}^{(m)}\}$ contains a subsequence which converges to $\nabla u_\epsilon$ weakly in $(L^{p(\cdot)+\delta}(Q_T))^N$, whence \eqref{eq:grad-high-eps}. \end{proof}

\begin{thm}
\label{th:exist-source-reg}
Let in the conditions of Theorem \ref{th:exist-reg}, $b \not\equiv 0$.
\begin{itemize}
\item[{\rm (i)}] Assume that {$b$, $\sigma$} are measurable and bounded functions in $Q_T$
\[
\|\nabla b\|_{\infty,Q_T}<\infty, \quad \|\nabla \sigma\|_{\infty,Q_T}<\infty, \qquad 2\leq \sigma^-\leq \sigma^+<1+\dfrac{p^-}{2}.
\]
Then for every $\epsilon\in (0,1)$ problem \eqref{eq:reg-prob} has at least one strong solution $u$, which satisfies estimates \eqref{eq:est-strong-eps-1}, \eqref{eq:grad-high-eps}.

\item[{\rm (ii)}] The solution is unique if either $\sigma \equiv 2$, or $b(z)\leq 0$ in $Q_T$ and $\sigma^-\geq 1$.
    \end{itemize}
\end{thm}

\begin{proof}
The proof is an imitation of the proof of Theorem \ref{th:exist-reg}. The estimates of Lemmas \ref{le:source-1}, \ref{le:source-2}, \ref{le:source-3} allow one to extract a subsequence $\{u_{\epsilon}^{(m_k)}\}$ with the convergence properties \eqref{eq:conv}. Let $u_\epsilon$ be the pointwise limit of the sequence $\{u_{\epsilon}^{(m_k)}\}$. We have to show that for every $\phi\in L^2(Q_T)$

\[
\int_{Q_T}|u_\epsilon^{(m_k)}|^{\sigma(z)-2}u_\epsilon^{(m_k)}\phi\,dz\to \int_{Q_T}|u_\epsilon|^{\sigma(z)-2}u_\epsilon\phi\,dz.
\]
The sequence $v_{m_k}=|u_{\epsilon}^{(m_k)}|^{\sigma(z)-2}u_{\epsilon}^{(m_k)}$ converges a.e. in $Q_T$ to $|u_\epsilon|^{\sigma(z)-2} u_\epsilon$ and is uniformly bounded in $L^{2}(Q_T)$ because

\[
\begin{split}
\int_{Q_T}v_{m_k}^2\,dz & =\int_{Q_T}|u_\epsilon^{(m_k)}|^{2(\sigma(z)-1)}\,dz\leq C\left(1+\int_{Q_T}|u_\epsilon^{(m_k)}|^{p^-}\,dz\right)
\\
& \leq C\left(1+\int_{Q_T}|\nabla u_\epsilon^{(m_k)}|^{p^-}\,dz\right)\leq C\left(1+\int_{Q_T}|u_\epsilon^{(m_k)}|^{p(z)}\,dz\right)\leq C'.
\end{split}
\]
 It follows that there is $v\in L^{2}(Q_T)$ such that $v_{m_k}\rightharpoonup v$ in $L^{2}(Q_T)$ and by virtue of pointwise convergence it is necessary that $v=|u_{\epsilon}|^{\sigma(z)-2}u_{\epsilon}$ a.e. in $Q_T$.

 Assume that $u_1, u_2\in \mathcal{W}_{q(\cdot)}(Q_T)$ are two strong solutions of problem \eqref{eq:reg-prob}. The function $u_1-u_2$ is an admissible test-function in the integral identities \eqref{eq:def} for $u_i$. Combining these identities and using \eqref{eq:mon} we arrive at the inequality
 \[
 \begin{split}
 \frac{1}{2}\|u_1-u_2\|_{2,\Omega}^2(t) & \leq \frac{1}{2}\|u_1-u_2\|_{2,\Omega}^2(t)+ \int_{0}^t\int_{\Omega}(\gamma_\epsilon(z,\nabla u_1)\nabla u_1-\gamma_\epsilon(z,\nabla u_2)\nabla u_2)\cdot \nabla (u_1-u_2)\,dz
 \\
 & =\int_{0}^t\int_{\Omega}b(z)\left(|u_1|^{\sigma(z)-2}u_1 -|u_2|^{\sigma(z)-2}u_2\right)(u_1-u_2)\,dz.
 \end{split}
 \]
 If $\sigma \equiv 2$, this inequality takes the form
 \[
 \frac{1}{2}\|u_1-u_2\|_{2,\Omega}^2(t)\leq B \int_{0}^t\|u_1-u_2\|_{2,\Omega}^2(\tau)\,d\tau,\quad t\in (0,T),\quad B=\operatorname{ess}\sup_{Q_T}b(z),
 \]
 whence $\|u_1-u_2\|_{2,\Omega}(t)=0$ in $(0,T)$ by Gr\"onwall's inequality. Let $b(z)\leq 0$ in $Q_T$. For $\sigma(z)\geq 1$ the function $|s|^{\sigma(z)-2}s$ is monotone increasing as a function of $s$, therefore $\left(|u_1|^{\sigma(z)-2}u_1 -|u_2|^{\sigma(z)-2}u_2\right)(u_1-u_2)\geq 0$ a.e. in $Q_T$ and
 \[
 \frac{1}{2}\|u_1-u_2\|_{2,\Omega}^2(t)\leq 0\quad \text{in $(0,T)$}.
 \]
\end{proof}
\subsection{Degenerate problem. Proof of Theorems  \ref{PPresu2}, \ref{th:exist-source}}
 Let $\{u_\epsilon\}$ be the family of strong solutions of the regularized problems \eqref{eq:reg-prob} satisfying estimates \eqref{eq:est-strong-eps-1}. These uniform in $\epsilon$ estimates enable one to extract a sequence $\{u_{\epsilon_k}\}$ and find functions $u\in {\mathcal{W}_{q(\cdot)}(Q_T)}$, $\eta,\,\chi\in (L^{q'(\cdot)}(Q_T))^N$ with the following properties:
\[
\begin{split}
& u_{\epsilon_k} \to u \quad \text{$\star$-weakly in $
L^\infty(0,T;L^2(\Omega))$}, \ \qquad \text{$u_{\epsilon_k t} \rightharpoonup u_{t}$ in
$L^2(Q_T)$},
\\
& \text{$\nabla u_{\epsilon_k} \rightharpoonup \nabla
u$  in $(L^{q(\cdot)}(Q_T))^N$},
\\
& \text{$(\epsilon_k^2 + |\nabla u_{\epsilon_k}|^2)^{\frac{p(z)-2}{2}} \nabla u_{\epsilon_k}\rightharpoonup \eta$ in
$(L^{q'(\cdot)}(Q_T))^N$}, \\
& \text{$(\epsilon_k^2 + |\nabla u_{\epsilon_k}|^2)^{\frac{q(z)-2}{2}} \nabla u_{\epsilon_k}\rightharpoonup {\chi}$ in
$(L^{q'(\cdot)}(Q_T))^N$}.
\end{split}
\]
In the third line we make use of the uniform estimate
\[
\int_{Q_T}(\epsilon^2+|\nabla u_{\epsilon}|^2)^{\frac{q(z)(p(z)-1)}{2(q(z)-1)}}\,dz\leq C\left(1+\int_{Q_T}|\nabla u_{\epsilon}|^{p(z)+r}\,dz\right)\leq C,
\]
which follows from \eqref{eq:oscillation} and \eqref{eq:grad-high-eps}. Moreover, $u\in C([0,T];L^2(\Omega))$. Each of $u_{\epsilon_k}$ satisfies the identity
\begin{equation}
\label{eq:ident-lim-k} \int_{Q_T} u_{\epsilon_k t} \xi \,dz +
\int_{Q_T} \gamma_{\epsilon_k}(z,\nabla u_{\epsilon_k})\nabla u_{\epsilon_k} \cdot \nabla \xi \,dz = \int_{Q_T} f_0 \xi
\,dz\qquad \forall \xi \in {\mathcal{W}_{q(\cdot)}(Q_T)},
\end{equation}
which yields\begin{equation}
\label{eq:ident-prelim}
\int_{Q_T} u_{ t} \xi \,dz +
\int_{Q_T} (\eta+ a(z) {\chi}) \cdot \nabla \xi \,dz = \int_{Q_T} f_0 \xi
\,dz\qquad \forall \xi \in {\mathcal{W}_{q(\cdot)}(Q_T)}.
\end{equation}
To identify $\eta$ and ${\chi}$ we use the monotonicity argument. Take $\xi=u_{\epsilon_k}$ in \eqref{eq:ident-lim-k}:
\begin{equation}
\label{eq:ident-last}
\int_{Q_T} u_{\epsilon_k t} u_{\epsilon_k} \,dz +
\int_{Q_T} \gamma_{\epsilon_k}(z,\nabla u_{\epsilon_k}) \nabla u_{\epsilon_k} \cdot \nabla u_{\epsilon_k}\,dz = \int_{Q_T} f_0 u_{\epsilon_k}
\,dz.
\end{equation}
According to \eqref{eq:mon-calc}, for every {$\phi\in \mathcal{W}_{q(\cdot)}(Q_T)$}
\[
\begin{split}
\int_{Q_T}  &\gamma_{\epsilon_k}(z,\nabla u_{\epsilon_k})\nabla u_{\epsilon_k}\cdot \nabla u_{\epsilon_k}\,dz
\geq  \int_{Q_T}(\gamma_{\epsilon_k}(z,\nabla \phi)-(|\nabla \phi|^{p-2}+ a(z) |\nabla \phi|^{q-2})\nabla \phi\cdot \nabla (u_{\epsilon_k}-\phi)\,dz
\\
& + \int_{Q_T} \gamma_{\epsilon_k}(z,\nabla u_{\epsilon_k})\nabla u_{\epsilon_k}\cdot \nabla \phi\,dz
+ \int_{Q_T}(|\nabla \phi|^{p-2}+ a(z) |\nabla \phi|^{q-2})\nabla \phi\cdot \nabla (u_{\epsilon_k}-\phi)\,dz
\\
&
\equiv J_{1,k}+J_{2,k}+J_{3,k},
\end{split}
\]
where
\[
\begin{split}
& J_{2,k}\to \int_{Q_T}(\eta+ a(z) {\chi})\cdot \nabla \phi\,dz,
\\
& J_{3,k}\to \int_{Q_T}(|\nabla \phi|^{p-2}+ a(z) |\nabla \phi|^{q-2})\nabla \phi\cdot \nabla (u-\phi)\,dz\quad \text{as $k\to \infty$}.
\end{split}
\]
Since $\left|(\gamma_{\epsilon_k}(z,\nabla \phi)\nabla \phi-(|\nabla \phi|^{p-2}+ a^{\frac{q-1}{q}}(z) |\nabla \phi|^{q-2}) \nabla \phi\right|\to 0$ a.e. in $Q_T$ as $k\to\infty$, and because the integrand of $J_{1,k}$ has the majorant
\begin{equation*}
\begin{split}
\left|((\epsilon_k^2 + |\nabla \phi |^{2})^\frac{p-2}{2}  -|\nabla \phi|^{p-2}) \nabla \phi \right|^{p'} &+ \left| a^{\frac{q-1}{q}}(z) ((\epsilon_k^2 + |\nabla \phi |^{2})^\frac{q-2}{2} - |\nabla \phi|^{q-2})\nabla \phi)\right|^{q'} \\
&\leq C \left(((1+|\nabla \phi|^2)^{\frac{p(z)}{2}} + a(z) (1+|\nabla \phi|^2)^{\frac{q(z)}{2}}\right)\\
&\leq C \left(1+|\nabla \phi|^{p(z)} + a(z)|\nabla \phi|^{q(z)}\right),
\end{split}
\end{equation*}
then $J_{1,k}\to 0$ by the dominated convergence theorem. Combining \eqref{eq:ident-prelim} with \eqref{eq:ident-last} and letting $k\to \infty$ we find that for every {$\phi\in \mathcal{W}_{q(\cdot)}(Q_T)$}
\[
\int_{Q_T}\left((|\nabla \phi|^{p(z)-2} + a(z) |\nabla \phi|^{q(z)-2})\nabla \phi- (\eta + a(z) {\chi} )\right)\cdot \nabla(u-\phi)\,dz\geq 0.
\]
Choosing $\phi=u+\lambda \zeta$ with $\lambda>0$ and $\zeta\in {\mathcal{W}_{q(\cdot)}(Q_T)}$, simplifying, and then letting $\lambda\to 0^+$, we obtain the inequality
\[
\int_{Q_T}\left((|\nabla u|^{p(z)-2} \nabla u + a(z) |\nabla u|^{q(z)-2} \nabla u) - (\eta + a(z) {\chi})\right)\cdot \nabla\zeta\,dz\geq 0\quad \forall \zeta\in {\mathcal{W}_{q(\cdot)}(Q_T)}.
\]
Since the sign of $\zeta$ is arbitrary, the previous relation is the equality. It follows that in \eqref{eq:ident-prelim} $\eta + a(z) \chi$ can be substituted by $|\nabla u|^{p(z)-2}\nabla u + a(z) |\nabla u|^{q(z)-2}\nabla u$. Since $u\in C([0,T];L^{2}(\Omega))$, the initial condition is fulfilled by continuity. Estimates \eqref{eq:strong-est} follow from the uniform in $\epsilon$ estimates of Theorem \ref{th:exist-reg} and the lower semicontinuity of the modular exactly as in the proof of Theorem \ref{th:exist-reg}. Uniqueness of a strong solution is an immediate consequence of the monotonicity. Theorem \ref{PPresu2} is proven.

\medskip

To prove Theorem \ref{th:exist-source} we only have to check that $|u_{\epsilon_k}|^{\sigma(z)-2}u_{\epsilon_k}\rightharpoonup |u|^{\sigma(z)-2}u$ in $L^2(Q_T)$ (up to a subsequence). This is done as in the case of the regularized problem.

{\begin{remark}
Under the assumption of the Theorem \ref{PPresu2} or Theorem \ref{th:exist-source} and, in addition $f_0 \in L^1(0,T; L^\infty(\Omega))$ and $u_0 \in L^\infty(\Omega)$, the strong solution of the problem \eqref{eq:main} is bounded and satisfies the estimate
\begin{equation}\label{est:bdd}
\|u(\cdot,t)\|_{\infty,\Omega} \leq e^{C_1t} \|u_0\|_{\infty, \Omega} + e^{C_1 t} \int_0^t e^{-C_1 \tau} \|f_0(\cdot, \tau)\|_{\infty, \Omega}~d\tau
\end{equation}
where $C_1=0$ if $b(z) \leq 0$ {in $Q_T$}, or $C_1= \|b\|_{\infty,Q_T}$ if ${\sigma} \equiv 2$ (see  \cite[Ch.4,Sec.4.3,Th.4.3]{ant-shm-book-2015}).
\end{remark}}

\end{document}